\documentclass[onecollarge]{svjour2}       % onecolumn "king-size"
\smartqed  % flush right qed marks, e.g. at end of proof
\usepackage{graphicx,color}
\usepackage{epsfig}
\usepackage{epstopdf}
\usepackage{amssymb}
\usepackage{amsmath}
\usepackage{tikz}
\usetikzlibrary{arrows}

\DeclareSymbolFont{largesymbol}{OMX}{yhex}{m}{n}
\DeclareMathAccent{\Widehat}{\mathord}{largesymbol}{"62}

\newcommand{\EEA}{\end{eqnarray}}
\newcommand{\BEA}{\begin{eqnarray}}

\newcommand{\peq}{p_{eq}}

\newcommand{\comment}[1]{}
\newcommand{\blue}[1]{{\color{black}#1}}
\newcommand{\td}{\textrm{d}}
\newcommand{\red}[1]{{\color{black}#1}}
\begin{document}

\title{A Parameter Estimation Method Using Linear Response Statistics%\thanks{Grants or other notes
}


\author{      
 John Harlim \and        Xiantao Li  \and He Zhang %etc.
}


\institute{John Harlim\at
             Department of Mathematics \\
             Department of Meteorology and Atmospheric Science \\
              The Pennsylvania State University \\
              \email{jharlim@psu.edu}           %  \\
           \and
           Xiantao Li \at
            Department of Mathematics \\
            The Pennsylvania State University  \\
              \email{xxl12@psu.edu}
	\and
	He Zhang\at
	      Department of Mathematics \\	
              The Pennsylvania State University \\
              \email{hqz5159@psu.edu} 
   }

\date{Received: date / Accepted: date}

\maketitle

\begin{abstract}
This paper presents a new parameter estimation method for It\^{o} diffusions such that the resulting model predicts the equilibrium statistics as well as the sensitivities of the underlying system to external disturbances. 
Our formulation does not require the knowledge of the underlying system, however we assume that the linear response statistics can be computed via the fluctuation-dissipation theory. The main idea is to fit the model to a finite set of ``essential" statistics that is sufficient to approximate the linear response operators. In a series of test problems, we will show the consistency of the proposed method in the sense that if we apply it to estimate the parameters in the underlying model, then we must obtain the true parameters. 

\end{abstract}

\section{Introduction}\label{intro}

Modeling complex dynamical systems usually involves proposing a model based on some physical laws, estimating the model parameters from the available observed data, and verifying the results against observables. An important practical issue in modeling is that all models are subjected to error, whether it is due to incomplete physical understanding, misspecification of parameters, or numerical discretization. Regardless of how one chooses the models, a reliable method is needed to determine the parameters, so that the model would have the desired predictive capability. Standard approaches for estimating parameters are to fit the model to the observed data using the maximum likelihood estimate \cite{pavliotis:14} or the Bayesian inference scheme such as the Markov-Chain Monte Carlo methods \cite{gamerman:06}. One can also estimate some of the parameters by fitting the model to some equilibrium statistics \cite{delsole:1996}. In this paper, we propose a new parameter estimation method which can predict the equilibrium statistics, with  additional considerations that will be discussed next.

An important issue raised in \cite{mg:11a} is whether the resulting model is able to recover the sensitivities of the underlying system to extenal disturbances, in addition to the equilibrium statistics. What they showed is that in the presence of model error, there are intrinsic barriers to improve the model sensitivities even when the equilibrium statistics are fully recovered. In subsequent works \cite{mg:11b,mq:16a,qm:16a}, they designed a parameter estimation scheme to recover both the equilibrium statistics and the model sensitivities, and applied their method successfully to various high-dimensional geophysical turbulent problems. Their parameter estimation method is formulated as a minimization problem utilizing an information criterion that involves fitting the linear response operators for the mean and variance. 

Motivated by the positive results in their papers \cite{mg:11b,mq:16a,qm:16a}, we propose to infer the parameters of the It\^{o} diffusions by using the linear response statistics. In our formulation, we assume that the underlying unperturbed dynamics are unknown and that the linear response statistics can be computed via the fluctuation-dissipation-theory (FDT) \cite{mag:05}. Implicitly, this means that all the hypothesis in \cite{hm:10} are satisfied and the equilibrium density of the unperturbed dynamics is known. Of course, the equilibrium density functions for high dimensional problems may not be easily estimated if they are not available as pointed out in \cite{mq:16a}. In that case, one may want to relax the assumption of not knowing the underlying unperturbed dynamics and apply the ``kicked" response theory to approximate the linear response operator \cite{mg:11b,mq:16a}. 

The main departure of our approach with the one proposed in \cite{mg:11b,mq:16a,qm:16a} is as follows. Their method involves minimizing an information-theoretic functional that depends on both the mean and variance response operators. \comment{Note that these operators are infinite dimensional objects.} Our idea is to define a finite number of quantities, to be referred to as the \emph{essential statistics}, which would allow us to estimate the linear response operators, which in principle are infinite dimensional objects, of appropriate observables (beyond just the mean and variance). To be more specific, we use a rational approximation to estimate the linear response statistics and the coefficients in the rational functions are related to the essential statistics. Subsequently we estimate the parameters in the model by fitting to these essential statistics. In our approach, we formulate the parameter inference problem into a problem of solving a system of nonlinear equations that involve these essential statistics. 

The main goal of this paper is to demonstrate that this approach is a consistent parameter estimation method. This means that if we use the proposed method to estimate the parameters of the underlying model, then we must obtain the true parameters. With this goal in mind, we analyze the proposed approach on three examples. The first example is linear and Gaussian. The second example is a nonlinear stochastic differential equation \cite{majda2016} which was introduced as a prototype model for geophysical turbulence. In these two examples, we will show that all the parameters can be computed explicitly and the consistency of the parameter estimation method can be established. The third example is a Langevin dynamics of a particle driven by conservative, dissipative, and stochastic forces. In this example, we shall see that some parameters cannot be computed analytically. Sufficient conditions to achieve consistent parameter estimation will be discussed. 

The remaining part of the paper is organized as follows. In Section~2, we briefly review the fluctuation-dissipation theory for estimating the linear response statistics. In Section~3, we present the proposed parameter inference method using the linear response statistics and established the consistency condition for linear and Gaussian dynamics. In Section~4, we show applications on the nonlinear examples. In Section~5, we conclude the paper with a brief summary and discussion. Two Appendices are included to show some computational details that are left out in the main text.

\section{A Brief Overview of the Fluctuation-dissipation theory (FDT)}

The fluctuation-dissipation theory (FDT) is a mathematical framework for quantifying the response of a dynamical system  in the linear regime subject to small external forcing (see e.g., \cite{leith:75,mag:05}). Consider an $n$-dimensional stochastic dynamical system,
\BEA
dx = a(x,\theta)\,dt +b(x,\theta)dW_t,\label{SDE}
\EEA
where $a(x,\theta)$ denotes a vector field, $b(x,\theta)$ denotes a diffusion tensor, $W_t$ denotes standard Wiener processes, and $\theta$ denotes the model parameter. Here, we assume that the solutions of \eqref{SDE} can be characterized by a density function $p(x,t)$. We also assume that the system in \eqref{SDE} is ergodic with equilibrium density $\peq(x)$. Consider an external perturbation of the following form, $f(x,t) = c(x)\delta f(t)$ of order-$\delta$ where $0<\delta\ll 1$, such that the solutions of the perturbed dynamics, 
\BEA
dx = \big(a(x,\theta) + c(x)\delta f(t)\big)\,dt +b(x,\theta)dW_t,\nonumber
\EEA
can be characterized by a perturbed density $p^\delta(x,t)$. By perturbation theory (see \cite{mag:05} for details), assuming that the system is initially at equilibrium, the difference between the perturbed and unperturbed statistics of any integrable function $A(x)$ can be estimated by a convolution integral up to order-$\delta^2$, 
\BEA
\delta \mathbb{E}[A](t) = \mathbb{E}_{p^\delta} [A(x)](t) - \mathbb{E}_{\peq} [A(x)] =  \int_0^t k_A(t-s)\delta f(s)\,ds + \mathcal{O}(\delta^2).\label{responsestat}
\EEA
In \eqref{responsestat}, the term $k_A(t)$ is known as the linear response operator. The FDT formulates the linear response operator as,
\BEA
k_A(t) = \mathbb{E}_{\peq}[A(x(t))\otimes B(x(0))],\quad\mbox{where}\quad [B(x)]_i = -\frac{\partial_{x_i} (c_i(x)\peq(x))}{\peq(x)}\label{RA},
\EEA
and $c_i(x)$ denotes the $i$th component of $c(x)$. 

If $A(x)=x$, then $k_A$ corresponds to the mean response operator and if { $A(x)=\big(x- \mathbb{E}(x)\big)^2$}, then $k_{A}$ corresponds to the covariance response operator. The form of this response operator is the same when the underlying dynamical system is a system of ordinary differential equations, assuming that the invariant density is differentiable with respect to $x$ \cite{gww:2016}. In statistical mechanics literature, FDT is known as the linear response approach \cite{Toda-Kubo-2,EvMo08}, which is the foundation for defining transport coefficients, e.g., viscosity, diffusion constant, heat conductivity etc. The important observation is that non-equilibrium properties can be estimated based on two-point equilibrium statistics.

In practice, the response operator $k_A(t)$ can be computed with a direct time average,
\BEA
k_A(t_j) \approx \frac{1}{N} \sum_{i=1}^N A(x_{i+j})\otimes B(x_i),\label{RAMC}
\EEA
given samples of the {\it unperturbed} dynamics, $x_i = x(t_i) \sim\peq(x)$, a parametric form of the invariant density $\peq(x)$, and the {external force} $f(x,t)=c(x)\delta(t)$. Thus, the FDT response offers a data-driven method, and it does not require knowing the model \eqref{SDE} nor samples of perturbed dynamics. Implicitly, this means that all the hypothesis in \cite{hm:10} are satisfied and the equilibrium density of the unperturbed dynamics is known. As pointed out in \cite{mq:16a}, the equilibrium density function for high dimensional problems may not be easily estimated if they are not available, \comment{ as pointed out in \cite{mq:16a}} which is beyond the scope of this paper. Our focus is to use the response statistics, defined based on the samples of the unperturbed dynamics, as a means to identify the model parameters $\theta$, which will be illustrated in the next section.

\section{The parameter estimation problem}

The goal of this paper is to develop a method to infer the parameters $\theta$ in the stochastic model in \eqref{SDE}, given the time series of $x_i =x(t_i;\theta)\sim\peq(x)$. Here we use a short notation to suppress the dependence on the parameter $\theta$. In particular, we assume that we have no access to the underlying dynamics in \eqref{SDE} in the sense that we know the model except the parameters that produce the data set $x_i$. While the focus of this paper is to determine parameters of the same model that generates the underlying truth, the proposed method below will be applicable even if the model to be parameterized is different than the underlying model that generates the data $x_i$. In this case, the method does not need to know the underlying model. To make this point clear, we consider a general setup where the main goal is
to estimate the parameters $\tilde{\theta}$ of the following system of SDE's,
\BEA
d\tilde{x} =\tilde a(\tilde x,\tilde\theta)\,dt + \tilde b(\tilde x,\tilde\theta)\,d\tilde{W}_t,\label{approxSDE}
\EEA
whose equilibrium density is characterized by $\tilde{p}_{eq}$. One of the main goals of this paper is to introduce consistent parameter estimation method. This means that if we employ the method to the underlying model, it will yield the true parameters. Mathematically, this consistency condition is achieved when $\tilde a(\tilde x,\theta) = a(\tilde x,\theta)$ and \red{ $\tilde b(\tilde x,\theta) \tilde b(\tilde x,\theta)^\top = b(\tilde x,\theta)b(\tilde x,\theta)^\top$}. Namely, when the exact parameters are recovered, the exact model will be recovered. As we will show in examples below in Section~4, in fact, some of these parameters are not identifiable by fitting the one-time equilibrium statistics alone, and additional statistics need to be included.   

Following the idea in \cite{mg:11a,mq:16a}, we consider inferring the parameters $\tilde\theta$ in \eqref{approxSDE} such that the linear response statistics of the model in \eqref{approxSDE},
\BEA
\tilde\delta\mathbb{E}[A] =  \mathbb{E}_{\tilde p^\delta} [A(\tilde{x})](t) - \mathbb{E}_{\tilde{p}_{eq}} [A(\tilde{x})], \label{approxresponsestat}
\EEA
under the external forcing, $f(x,t)=c(x)\delta f(t)$, agree with the linear response statistics $\delta \mathbb{E}[A]$ given in \eqref{responsestat} of the underlying system \eqref{SDE} under the same external disturbances. In \eqref{approxresponsestat}, the notation $\tilde p^\delta$ denotes the perturbed density of the approximate model in \eqref{approxSDE} under external forcing. With this setup, let us illustrate our idea for solving this inference problem.

\subsection{The essential statistics}
The key challenge in this inference problem is to find the parameters such that the resulting response operator in \eqref{approxresponsestat} agree with $k_A(t)$ without knowing the underlying dynamics in \eqref{SDE}.
It is clear that computing the full response kernel $k_A(t)$ for each $t\ge 0$ is not practical, since it is an infinite dimensional object. Instead, we define a finite number of quantities, to be referred to as the \emph{essential statistics}, which would allow us to approximate the linear response operator $k_A(t)$, and subsequently estimate the parameters in the model by fitting to these essential statistics.

For this purpose, we define
\BEA
y(t) = \int_0^t k_A(t-\tau)\delta f(\tau)\,d\tau. \label{yint}
\EEA
Taking the Laplace transform on both sides, we have, 
\BEA
Y(s) = K(s)\delta F(s),\label{Y=KS}
\EEA
where $Y, K, \delta F$ denote the Laplace transforms of $y, k_A, \delta f$, respectively. Next consider a rational approximation in terms of $s^{-1}$,
\BEA
K(s) \approx R_{m,m} \overset{\rm def}{=}(I-s^{-1}\beta_1-s^{-m}\beta_m)^{-1} (s^{-1}\alpha_1 + \ldots + s^{-m}\alpha_m), \label{rationalapprox}
\EEA
where $\alpha_i, \beta_i \in \mathbb{R}^{n\times n}$ are to be determined. Compared to the standard polynomial approximations, the rational approximation often offers a larger radius of convergence. More importantly, one can choose parameters so that the approximation has appropriate asymptotes toward infinity. 

To illustrate the idea, suppose $m=1$, and we have, 
\begin{equation}
K(s) \approx R_{1,1} \overset{\rm def}{=} [I-\lambda \beta_1 ]^{-1} \lambda \alpha_1, \quad \lambda=\frac1s.\label{R11}
\end{equation}
To determine the coefficients, we expand $R_{1,1}$ around $\lambda = 0_+,$
 \[ R_{1,1} = \lambda \alpha_1 + \lambda^2 \beta_1\alpha_1 + \mathcal{O}(\lambda^3).\]
Similarly, we expand the response kernel, and with direct calculations, we find that,
\[ K(s) = \int_0^t e^{-t/\lambda} k_A(t) dt = \lambda k_A(0) + \lambda^2 k_A'(0) +  \mathcal{O}(\lambda^3),\] which can be obtained using repeated integration-by-parts. 
By matching the first two coefficients, which is a standard Pad\'e approximation, we arrive at
 \BEA
 \alpha_1=k_A(0),\quad \beta_1= k_A'(0)  k_A(0)^{-1}.\nonumber
\EEA

The corresponding approximation in the time domain can be deduced as follows.  Substituting \eqref{R11} to \eqref{Y=KS}, we get, 
\BEA
sY(s) = \beta_1 Y(s) + \alpha_1\delta F(s),\nonumber
\EEA
Now noticing that $R_{1,1}(0)=k_A(0)$ due to the first matching condition, we can apply the inverse Laplace transform, and obtain  the following initial value problem,
\BEA
y' = \beta_1 y + \alpha_1\delta f, \quad y(0) = 0,\nonumber
\EEA
or equivalently,
\BEA
y(t) = \int_0^t e^{(t-\tau) \beta_1}\alpha_1\delta f(\tau)\,d\tau. \label{approxyint}
\EEA
Comparing \eqref{yint} and \eqref{approxyint}, we see that this procedure corresponds to approximating the response operator by,
\BEA
k_A(t) \approx e^{t \beta_1}\alpha_1\label{approxkA},
\EEA
where $e^{t \beta_1}$ is a matrix exponential.

For arbitrary $m$, the rational function approximation of the Laplace transform corresponds to the following  approximation in the time domain,
\BEA
k_A(t) \approx g_m(t) \overset{\rm def}{=}\begin{pmatrix}I & 0 & \cdots & 0\end{pmatrix} e^{t G
}\begin{pmatrix}\alpha_1 \\ \alpha_2 \\ \vdots \\ \alpha_m \end{pmatrix}, \quad G =\begin{pmatrix}
\beta_1 & I  & 0 &\cdots &0& \\ 
\beta_2 & 0 & I &\cdots &0&\\
\vdots & \vdots & \vdots & \ddots & I &\\
\beta_m & 0 & 0& \cdots  & 0
\end{pmatrix}.\label{eq: ansatz}
\EEA
To determine the coefficients, we extend the first order approximation and enforce additional conditions. More specifically, we choose $t_0 < t_1 < \cdots < t_j,$ and enforce the conditions,
\BEA\label{multi-pade}
  k_A^{(i)} (t_0) = r^{(i)} (t_0),  \quad  k_A^{(i)} (t_1) = r^{(i)} (t_1), \quad \cdots,  \quad k_A^{(i)} (t_{J}) = r^{(i)} (t_J), \quad 0\le i \le 2m-1.
\EEA
Here $r(t)$ is the inverse Laplace transform of the rational function $R_{m,m}(\lambda).$ \blue{Also, if $J=0$, then we have the same number of equations and parameters, $2m$, and the approximation becomes a standard Pad\'e approximation at $\lambda=0.$ }
The %single-point approximation $J=0$ with $t_0=0$ corresponds to a standard Pad\'e approximation, and the 
matching conditions involve the derivatives of the response function at zero: $k^{(i)}_A(0), i =0,\ldots, 2m-1$.  If $J>1$, we would have more equations than the unknowns.
Based on these observations, we define:

\begin{definition}
Let $\{ k_A^{(i)}(t_j)\}_{j=0,1,\ldots, J}$ be the {\it essential statistics} in a sense that they are sufficient to approximate the response operator $k_A(t)$ up to order-$m$, for $i =0,\ldots, 2m-1$.  
\end{definition}

\blue{In our discussion below, we will first focus mostly on the essential statistics at $t_0=0$ with $J=0$.} In particular, we define $M_i \overset{\rm def}{=}  k_A^{(i)}(0).$ The short-time essential statistics can be approximated, e.g., by finite difference approximation on the FDT response operator estimated in \eqref{RAMC}. \blue{Meanwhile, as we will demonstrate later, for problems where the intermediate time scale is also of importance, it is necessary to extend the essential statistics to different times $t_j\neq 0$, e.g., by using nonlinear least-square fitting to approximate the FDT response operator, and consider a more general matching procedure, such as the multipoint % Pad\`e
      approximation \eqref{multi-pade}.
}
 %As we will show below, the higher order derivatives can be difficult to estimate if the equilibrium variances of the unperturbed dynamics are small. 

\subsection{Inference method by fitting the essential statistics}
Given any integrable function $A(x)$, we define the response property of the surrogate model,
\BEA
\hat{k}_A(t,\tilde\theta) = \mathbb{E}_{\tilde{p}_{eq}}[A(x(t)) \otimes B(x(0))],\quad\mbox{where}\quad
B_i(x) = -\frac{\partial_{x_i} (c_i(x)\peq(x))}{\peq(x)}.\label{hatkA}
\EEA
We should stress that this is not the FDT response of \eqref{approxSDE} since $B$ is defined with respect to $\peq$  
as in \eqref{RA}. For the FDT response, $B$ should  be defined with respect to $\tilde{p}_{eq}$.  

We first consider estimating the parameters in \eqref{approxSDE} by solving 
\BEA
\hat{k}^{(i)}_A(0,\tilde\theta) = M_i, \quad i = 0, 1, \ldots, 2m-1.\label{mainfitting}
\EEA
By solving this system of equations, we essentially match the same statistics, estimated by averaging over two different realizations. One from the observed data which are samples of the unperturbed equilibrium solutions, and another one from the realizations of the approximate dynamics \eqref{approxSDE}. We should point out that for \red{$i=0$}, we are essentially fitting the equilibrium covariance statistics. 

Before discussing complex examples, let us establish the consistency of the proposed parameter estimation method in a simple context.

\begin{theorem}
Consider $x\in\mathbb{R}^n$ that solves a system of linear SDEs,
\BEA
dx = Cx\,dt + D\,dW_t, \quad x(0)=x_0,\label{linearSDE}
\EEA
where $W_t$ is a $d$-dimensional Wiener process and matrices $C\in\mathbb{R}^{n\times n}$ and $D\in\mathbb{R}^{n\times d}$ are defined such that $x$ is ergodic. The underlying parameters $C$ and $DD^\top$ in \eqref{linearSDE} can be estimated exactly by solving a system of equations in \eqref{mainfitting} with $m=1$, where $k_A$ and $\hat{k}_A$ are operators defined with $A(x)=x$ and constant external forcing $f(x,t) = \delta$.
\end{theorem}

\begin{proof}
Notice that the system in \eqref{linearSDE} has a unique Gaussian invariant measure with a symmetric positive definite covariance matrix $S$ that solves the Lyapunov equation $CS+SC^\top+DD^\top=0$. In this problem, the given model is exactly in the form of \eqref{linearSDE}, that is, $d\tilde x = \tilde C\tilde{x}\,dt + \tilde{D}\,dW_t$ and our goal is to employ the fitting strategy proposed in \eqref{mainfitting} to estimate $\tilde{C}$ and $\tilde{D}$. For a constant forcing $f(x,t) = \delta$ and functional $A(x)=x$, one can verify that,
\BEA
k_A(t) = e^{tC}, \quad \hat{k}_A(t) = e^{t\tilde{C}}\tilde{S}S^{-1}. \nonumber
\EEA
Here $\tilde{S}$ is the equilibrium covariance matrix of the approximate model which solves the Lyapunov equation $\tilde{C}\tilde{S}+\tilde{S}\tilde{C}^\top+\tilde{D}\tilde{D}^\top=0$. Equating $M_0= \hat{k}_A(0)$ gives $\tilde{S}=S$, and setting $M_1=\hat{k}_A'(0)$ gives $C=\tilde{C}$. From the Lyapunov equation, it is clear that $\tilde{D}\tilde{D}^\top =- \tilde{C}\tilde{S}-\tilde{S}\tilde{C} = -CS-SC^\top = DD^\top$. \qed
\end{proof}

From this example, as well as the examples in the next section, one can see that it is important to define $B_i(x)$ according to  \eqref{hatkA}. In particular, since $p_{eq}$ is involved, $\hat{k}_A(0)$ leads to an equation that naturally connects the equilibrium covariance matrices. The key point of the Theorem above is that one can recover the true parameters by just applying the proposed fitting scheme in \eqref{mainfitting} with $m=1$. It also confirms that by fitting only the one-time equilibrium variances, $M_0= \hat{k}_x(0)$, we only obtain $\tilde{S}=S$ which is not sufficient to determine the model parameters. 

While our strategy is to fit the slope of the correlation function at $t=0$, alternatively, one can also fit the correlation times, 
\BEA
M_\infty := \int_0^\infty k_A(t)\,dt = \int_0^\infty \hat{k}_A(t)\,dt,\nonumber 
\EEA
to obtain $\tilde{C}=C$ and use the same argument above to obtain $\tilde{D}\tilde{D}^\top=DD^\top$. This type of statistics has a close connection with the widely used Green-Kubo formula in statistical physics \cite{kubo1966fluctuation,Toda-Kubo-2}. In practice, however, this fitting procedure may encounter the following problems: (1)  A long time series is usually needed; (2)  It requires accurate estimation of the correlation functions $k_x(t)$ for $t\gg 1$ which can be computationally demanding, unless the correlation length is very short. Due to these potential issues, this approach won't be pursued here. 

For general nonlinear problems, the number of essential statistics, $m$, that is needed in the proposed fitting strategies depends on the properties of $\hat{k}^{(i)}_A(0,\tilde{\theta})$ as functions of the model parameters, $\tilde{\theta},$ and the number of model parameters. As we will see in the next section, the functional dependence on $\tilde{\theta}$ can be very nontrivial even in the case when the analytical expressions are available. \blue{We should also mention that numerically the higher order derivatives, $M_i=k^{(i)}_A(0), i>1$ will be difficult to estimate, as we shall see in Section 4.2 below. To mitigate this practical issue, we will consider matching the lower order derivatives at different points $t_j$ in placed of matching the higher order derivatives, $i>1$, in \eqref{mainfitting}. In particular, we will consider fitting the following finite number of essential statistics,
\BEA
\hat{k}^{(i)}_A(t_j,\tilde\theta) = k^{(i)}_A(t_j), \quad i = 0, 1, \quad 0=t_0<t_1< \ldots<t_J.\label{mainfitting2}
\EEA
We shall see that the fitting strategy in \eqref{mainfitting2} is a consistent parameter estimation method for nontrivial complex dynamical systems.

}  

\section{Nonlinear examples}

In this section, we demonstrate the parameter estimation procedure using two examples. The first example is a system of three-dimensional SDEs that characterizes geophysical turbulence. The second one is the Langevin equation, which is a classical example in statistical mechanics for the dynamics of a particle driven by various forces.

\subsection{A simplified turbulence model}

Consider the following system of SDEs, which were introduced in Chapter 2.3 of \cite{majda2016}  as a simplified model for geophysical turbulence,
\BEA
\frac{dx}{dt} = B(x,x) + Lx -\Lambda x + \sigma\Lambda^{1/2} \dot{W},\label{triad}
\EEA
where $x\in \mathbb{R}^3$. The term $B(x,x)$ is bilinear, satisfying the Liouville property, $div_x(B(x,x)) = 0$ and the energy conservation, $x^\top B(x,x) = 0$. The linear operator $L$ is skew symmetric, $x^\top L x = 0$, representing the $\beta$-effect of Earth's curvature. In addition, $\Lambda>0$ is a fixed positive-definite matrix and the operator $-\Lambda x$ models the dissipative mechanism. The white noise stochastic forcing term $\sigma\Lambda^{1/2}\dot{W}$ with scalar $\sigma^2>0$ represents the interaction of the unresolved scales.

As a specific example, we consider the following bilinear form for $B(x,x)$,
\BEA
B(x,x) = (B_1 x_2x_3,B_2 x_1x_3,B_3 x_1x_2)^\top\nonumber
\EEA
such that $B_1+B_2+B_3=0$ so that the Liouville and energy conservation conditions are satisfied. For this model, it is not very difficult to verify that the equilibrium density is Gaussian, \BEA\peq(x) \propto \exp\Big(-\frac{1}{2}\sigma_{eq}^{-2}x^\top x\Big),\label{peq}\EEA
where $\sigma_{eq}^2 = \sigma^2/2$. It is clear that the equilibrium statistics only depends on $\sigma_{eq}, $ and hence the parameters $\{B_1,B_2,B_3,L, \Lambda\}$ can't be estimated from fitting the one-time equilibrium statistics alone. Here we demonstrate how the essential statistics can help to estimate these parameters.
\smallskip

Let $f=\delta$ be a constant forcing, where $|\delta| \ll 1$ and suppose that the observable is $A(x) = x$. In this case, one can verify that the mean FDT response operator is the correlation function of $x$,
\BEA
k_A(t) = \frac{1}{\sigma_{eq}^2} \mathbb{E}_{\peq}[x(t)x(0)^\top]. \label{Rx}
\EEA  
This linear response operator can be estimated using the solutions of the unperturbed dynamics in \eqref{triad} (as shown  in Figure~\ref{fig1} with the gray curves, which we refer as the true response operator). In the numerical experiment, we have set $B_1=.5, B_2=1, B_3=-1.5$,  
\BEA
L = \begin{pmatrix}0 & 1 & 0 \\ -1 & 0 & -1 \\ 0 & 1 & 0 \end{pmatrix}, \quad\Lambda = \begin{pmatrix} 1\; & \frac12\; & \frac14\; \\  \frac12 & 1 & \frac12 \\ \frac14 & \frac12 & 1\end{pmatrix},\nonumber
\EEA
and $\sigma=1/5$. We solve the SDE to generate samples at  the equilibrium state with the weak trapezoidal method \cite{weak-trapezoildal-method} with step size $\delta t=2\times 10^{-4}$. The correlation function in \eqref{Rx} is computed using $10^7$ samples obtained by subsampling every 5 steps of the solutions up to 5000 model unit time. % without truncating the transient time.

With the data at hand, now consider estimating the parameters in \eqref{triad}. In other words, we assume that the underlying physical model is known except for the true parameters $\theta$. Using the `tilde' notation in \eqref{approxSDE}, we have at most $\tilde{\theta} \in \mathbb{R}^{12}$: two of $\tilde{B}_j$ since the third component can be specified by the energy conserving constraint ($\sum_j \tilde{B}_j = 0$); three upper off-diagonal components of the skew symmetric $\tilde{L}$; six upper diagonal components of symmetric $\tilde{\Lambda}$, and finally,  $\tilde{\sigma}$. %We require that the invariant measures coincide, $\tilde{p}_{eq} = \peq$.
Since the equilibrium density has zero mean, we don't need to fit the mean statistics. Furthermore, since the first two moments of any Gaussian statistics are the sufficient statistics, all of the higher-order statistics of odd order are zero and those with
 even order are automatically satisfied when the equilibrium covariance statistics are accounted (by matching $\hat{k}_x(0,\theta)=M_0$).

In Figure~\ref{fig1}, we show the order-one approximation \eqref{approxkA} of the mean response operator, i.e., $k_x(t)\approx g_1(t)= e^{t\beta_1}\alpha_1$ in \eqref{Rx} (black curve). Here, the parameters $\alpha_1,\beta_1$ are obtained by solving the equations, $M_0=\alpha_1$ and $M_1=\beta_1\alpha_1$.  Clearly, we observe very good agreement between the exact response functions and the order-1 approximation using just $M_0$ and $M_1$. \blue{For this example, we found that $M_0=I$ and $M_1=L-\Lambda$, where the derivation for $M_1$ is deduced below, cf. \eqref{kslope}. The exact value of $M_1$  given by, 
\begin{equation*}
M_1=\left(\begin{array}{ccc}
-1 & 0.5 & -0.25\\
-1.5 & -1.0 & -1.5\\
-0.25 & 0.5 & -1\\
\end{array}\right).
\end{equation*} 
On the other hand, a straightforward finite difference approximation based on the time correlation function computed from the numerical solution of the SDEs yields the following approximation,
\begin{equation*}
M_1\approx \left(\begin{array}{ccc}
-0.9987 & 0.51956 & -0.2416\\
-1.5088 & -0.9999 & -1.4976\\
-0.2623 & 0.5096 & -0.9979\\
\end{array}\right),
\end{equation*}
which is quite satisfactory compared to the exact values. This approach would be needed in cases where the time series of the solution (data) are available, but the full model is not. 
This suggests that by fitting the model to $M_0$ and $M_1$, we would capture the correct response property. 
}

\begin{figure}[htpb]
\centering
\includegraphics[width=.8\textwidth]{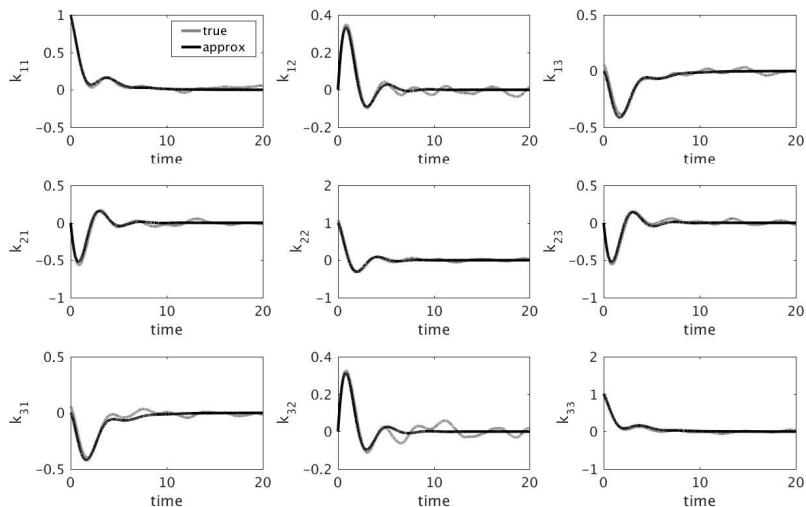}
\caption{Response operator of the three-dimensional system in Example~1: truth (dark) vs the first-order rational approximation (gray).}
\label{fig1}
\end{figure}

\smallskip

Now let's  return to the parameter estimation problem, by applying the fitting strategy in \eqref{mainfitting} to the triad model in \eqref{triad}. 
%\begin{example}
%[]%Revisiting Example 1.]
From \eqref{Rx} and the definition of $\hat{k}_x(t)$ in \eqref{hatkA}, it is clear that,
\BEA
M_0 = k_x(0) = I, \quad\mbox{and}\quad \hat{k}_x(0) = \frac{1}{\sigma_{eq}^2}\mathbb{E}_{\tilde{p}_{eq}}[x(0)x(0)^\top] = \frac{\tilde{\sigma}_{eq}^2}{\sigma_{eq}^2}I.\nonumber
\EEA
Here, we suppressed the parameter dependence by writing $\hat{k}_x(0) =\hat{k}_x(0,\tilde{\theta})$ and understood that $\tilde{\theta}=(\tilde{\sigma}, \tilde{B}_j, \ldots)\in\mathbb{R}^{12}$. Setting $M_0=\hat{k}_x(0)$ gives 
\BEA
\tilde{\sigma}_{eq}^2 = \frac{\tilde{\sigma}^2}{2}=\sigma_{eq}^2, \label{M0result}
\EEA
or equivalently, $\tilde{p}_{eq}=\peq$, since the invariant measure is Gaussian with mean zero and covariance $\sigma_{eq}^2I$ as noted in \eqref{peq}.

%From the FDT response operator, $k_x(t)$, we can employ the finite difference approximation to obtain the essential statistics $M_i$, $i\geq 1$. 
To compute $\hat{k}^{(i)}(0) \;\text{for}\; i\geq 1$, we let $\tilde{p}(x,t|y,0)$ be the solution of the corresponding Fokker-Planck equation \cite{risken1984fokker},
\BEA
\frac{\partial}{\partial t} \tilde{p} = \mathcal{\tilde{L}}^* \tilde{p}, \quad \tilde{p}(x,0) = \delta(y-x),\label{fp}
\EEA
where the operator $\mathcal{\tilde L}^*$ is the forward operator of the model whose parameters are to be estimated. Since this example concerns with estimating the parameters in the true model, the operator $\mathcal{\tilde{L}^*}$ is nothing but the forward operator of the true model in \eqref{triad}. In particular, the generator,  
\BEA \label{genapproxSDE1}
\tilde{\mathcal{L}} = (\tilde{B}(x,x)+\tilde{L}x-\tilde{\Lambda} x) \cdot \nabla + \frac{1}{2} \tilde{\sigma}^2\tilde{\Lambda}: \nabla^2
\EEA 
is the adjoint of $\mathcal{\tilde L}^*$\comment{ with respect to $L^2(\mathbb{R}^3)$}. Here, the operation $A:B = \sum_{i,j} A_{i,j}B_{i,j}$ and $\nabla^2$ denotes the Hessian operator. Using \eqref{fp} and \eqref{genapproxSDE1}, we can deduce that,
\BEA
\hat{k}'_x(t) &=& \frac{1}{\sigma_{eq}^2} \int_{\mathbb{R}^3} \int_{\mathbb{R}^3} A(x) \otimes B(y) \frac{\partial}{\partial t}\tilde{p}(x,t|y,0)\,dx\,dy
\nonumber \\ 
&=&\frac{1}{\sigma_{eq}^2} \int_{\mathbb{R}^3} \int_{\mathbb{R}^3} A(x) \otimes B(y) \mathcal{\tilde L}^*\tilde{p}(x,t|y,0)\,dx\,dy \nonumber \\
&=&\frac{1}{\sigma_{eq}^2} \int_{\mathbb{R}^3} \int_{\mathbb{R}^3} \mathcal{\tilde L} A(x) \otimes B(y) \tilde{p}(x,t|y,0)\,dx\,dy. \nonumber
\EEA
Since $A(x) = x, B(y) = \frac{1}{\sigma_{eq}^2}y$, and $\mathcal{\tilde{L}}$ is a differential operator with respect to $x$, by setting $t=0$, we obtain,
\BEA
\hat{k}'_x(0) 
&=& \frac{1}{\sigma_{eq}^2}\int_{\mathbb{R}^3} \int_{\mathbb{R}^3} (\tilde{B}(x,x)+\tilde{L}x-\tilde{\Lambda} x) y^\top \tilde{p}(x,0|y,0)\,dx\,dy. \nonumber \\
&=& \frac{1}{\sigma_{eq}^2}\int_{\mathbb{R}^3} (\tilde{B}(x,x)+\tilde{L}x-\tilde{\Lambda} x) x^\top \tilde{p}_{eq}(x)\,dx. \nonumber \\ &=&  
\frac{\tilde{\sigma}^2_{eq}}{\sigma^2_{eq}}(\tilde{L} - \tilde{\Lambda}).\label{kslope}
\EEA 
In deriving this, we have used the fact that the moments of order-3 of a Gaussian distribution with mean zero are zero. Setting $M_1=\hat{k}'_x(0)$, we obtain a system of nine equations,
\BEA
M_1 = \frac{\tilde\sigma_{eq}^2}{\sigma_{eq}^2}(\tilde{L} -\tilde{\Lambda})\label{M1result}.
\EEA 
Notice that with the equations \eqref{M0result} and \eqref{M1result} together,  we are still short by two equations if we want to estimate all the twelve parameters in the triad model. Using the same technique, one can find a set of equations $M_2=\hat{k}''_x(0)$, where the right hand terms are functions of $\tilde{B}_i$, the coefficients of the quadratic terms. In particular, continuing from the previous step of the calculation, we have,
\BEA
\hat{k}''_x(0) &=& \frac{1}{\sigma_{eq}^2}\int_{\mathbb{R}^3}\int_{\mathbb{R}^3}\mathcal{\tilde{L}}(\tilde{B}(x,x)+\tilde{L}x-\tilde{\Lambda} x) y^\top \tilde{p}(x,0|y,0)\td x\td y. \nonumber
\EEA
Direct calculations (see Appendix A) yields, 
\BEA\label{M2result}
M_2=\hat{k}''_x(0)=\frac{\tilde{\sigma}^{2}_{eq}}{\sigma_{eq}^{2}}(\tilde{L}-\tilde{\Lambda})^2-\frac{\tilde{\sigma}^{4}_{eq}}{\sigma^2_{eq}}\left(\begin{array}{ccc}
\tilde B^2_{1} & 0 & 0\\
0 & \tilde B^2_{2} & 0\\
0 & 0 & \tilde B^2_{3}\\
\end{array}\right).
\EEA
The coefficients $\tilde{B}_i$s are involved in $M_2$. However the signs of $\tilde B_{i}$ cannot be uniquely determined from \eqref{M2result}, since if $(\tilde B_1,\tilde B_2,\tilde B_3)$ satisfies the equation then $-(\tilde B_1,\tilde B_2,\tilde B_3)$ will also satisfy the equation. %And we end up with non-unique solutions for $\tilde B_i$. 
To identify uniquely the signs of the parameters $\tilde B_i$'s, a natural idea will be to check the third or higher order derivatives of $\hat{k}_x(t)$ at point $t=0$.  
%Unfortunately, the following proposition shows that we can not identify the signs of $\tilde B_i$s using such higher order statistics, i.e., by solving $M_i=\hat{k}^{(i)}_x(0)$ for $\forall i\geq 3$.

Let's verify whether this is a viable approach. For this purpose, we define some multi-index notation. Consider ${x}=(x_1,x_2,x_3)$ and 
$\alpha=(\alpha_1,\alpha_2,\alpha_3)$. We define ${x}^{{\alpha}}=x_1^{\alpha_1}x_2^{\alpha_2}x_3^{\alpha_3}$ and $|\alpha|=\alpha_1+\alpha_2+\alpha_3$. We will need the following lemma.

\begin{lemma}
Let $\mathcal{\tilde{L}}$ be the generator in \eqref{genapproxSDE1}, then repeated applications of the operator $\mathcal{\tilde{L}}$ to $x$ yield polynomials of $x$,%for $n\geq 1$ $\mathcal{L}^{n}(x)$ has the formula
\begin{equation}
\mathcal{\tilde L}^{n}x=\sum_{|{\alpha}|\leq n+1} P_{{\alpha}}{x}^{{\alpha}},\quad\quad  \forall n\geq 1.\label{Lnx}
\end{equation}
Moreover, if $|{\alpha}|$ is an odd number, then $P_{{\alpha}}=P_{{\alpha}}(\tilde{B}_1^2,\tilde{B}_2^2,\tilde{B}_1\tilde{B}_2,\tilde{\sigma},\tilde{L}-\tilde{\Lambda})$, that is, the coefficients of $P_{{\alpha}}$ only depend on the set of parameters $\{\tilde{B}_1^2,\tilde{B}_2^2,\tilde{B}_1\tilde{B}_2,\tilde{\sigma},\tilde{L}-\tilde{\Lambda}\}$.
\end{lemma}
\begin{proof}
To prove the lemma, we show that for all $n$, the coefficients in \eqref{Lnx} satisfy:

\begin{equation}
\begin{split}\label{coef}
P_{{\alpha}}=\begin{cases} P_{{\alpha}}(\{\tilde{B}^{{\beta}}\}_{|{\beta}|\leq |{\alpha}|-1,|{\beta}| \mbox{ is odd}},\tilde{\sigma},\tilde{L}-\tilde{\Lambda}),\quad \forall\;  |{\alpha}| \; \text{even}, \\
P_{{\alpha}}(\{\tilde{B}^{{\beta}}\}_{|{\beta}|\leq |{\alpha}|-1,|{\beta}| \mbox{ is even}},\tilde{\sigma},\tilde{L}-\tilde{\Lambda}),\quad \forall \; |{\alpha}|\;\text{odd},
\end{cases}
\end{split}
\end{equation}
where ${\tilde B}^{\beta} = \tilde{B}_{1}^{\beta_1}\tilde{B}_{2}^{\beta_2}\tilde{B}_{3}^{\beta_3}$ is also a multi-index notation. Equation~\eqref{coef} basically says that if $|\alpha|$ is even (odd), then $P_\alpha$ is the coefficient of $x^\alpha$ that depends on $\tilde{\sigma}$, $\tilde{L}-\tilde{\Lambda}$, $\{\tilde{B}^{{\beta}}\}_{|{\beta}|\leq |{\alpha}|-1}$, where $|{\beta}|$ is odd (even). That is, when $|{\alpha}|$ is odd (as in the statement of the lemma),
\begin{equation}
P_{{\alpha}}=P_{{\alpha}}(\{\tilde B_i\tilde B_j\},\tilde{\sigma},\tilde L-\tilde{\Lambda}). \nonumber
\end{equation}
Furthermore, since  $\tilde B_3=-(\tilde B_1+\tilde B_2)$ we end up with $P_{{\alpha}}=P_{{\alpha}}(\tilde B_1^2,\tilde B_2^2,\tilde B_1\tilde B_2,\tilde{\sigma},\tilde L-\tilde{\Lambda})$. 

Now we prove that the coefficients of \eqref{Lnx} satisfy equation~\eqref{coef} for all $n$. By induction,
\begin{enumerate}
\item For $n=1$, we have $\tilde{\mathcal{L}}x=\tilde B(x,x)+(\tilde L-\tilde {\Lambda})x$, and the coefficients satisfy \eqref{coef} since $(\tilde{L}-\tilde{\Lambda})x$ is the odd power term in $x$ and $\tilde{B}(x,x)$ is the even power term in $x$.
\item Assume for $n=k$, we set
\begin{equation*}
\tilde{\mathcal{L}}^{k}x=\sum_{|{\alpha}|\leq k+1} P^{k}_{{\alpha}}{x}^{{\alpha}}, \quad \tilde{\mathcal{L}}^{k+1}x=\sum_{|{\alpha}|\leq k+2} P^{k+1}_{{\alpha}}{x}^{{\alpha}}, \nonumber
\end{equation*}
where $P^k_\alpha$ satisfies \eqref{coef}. 
From $(\ref{genapproxSDE1})$ we have
\begin{equation*}
\begin{split}
\tilde{\mathcal{L}}P^{k}_{{\alpha}}x^{{\alpha}}
& = P^{k}_{{\alpha}}(\tilde{B}(x,x)+\tilde{L}x-\tilde{\Lambda}x)\cdot\nabla x^{{\alpha}}+\frac{1}{2}P^{k}_{{\alpha}}\tilde{\sigma}^{2}\tilde{\Lambda}:\nabla^2 x^{{\alpha}}\\
& = P^{k}_{{\alpha}}\tilde{B}(x,x)\cdot\nabla x^{{\alpha}}+P^{k}_{{\alpha}}(\tilde{L}-\tilde{\Lambda})x\cdot\nabla x^{{\alpha}}+\frac{1}{2}P^{k}_{{\alpha}}\tilde{\sigma}^{2}\tilde{\Lambda}:\nabla^2 x^{{\alpha}} \\
\end{split}
\end{equation*}
where the gradient vector $\nabla x^{{\alpha}}$ is a row vector. Notice $P^{k}_{{\alpha}}\tilde{B}(x,x)\cdot\nabla x^{{\alpha}}$ is of power $|{\alpha}|+1$, $P^{k}_{{\alpha}}(\tilde{L}-\tilde{\Lambda})x\cdot\nabla x^{{\alpha}}$ is of power $|{\alpha}|$ and $\frac{1}{2}P^{k}_{{\alpha}}\tilde{\sigma}^{2}\tilde{\Lambda}:\nabla^2 x^{{\alpha}}$ is of power $|{\alpha}|-2$. Thus, $\mathcal{\tilde L}^{k+1}x=\mathcal{\tilde L}\mathcal{\tilde L}^kx$ will at most have polynomial of degree $k+2$. In conclusion, we have the following table which shows the even or odd power terms of $x$ in $\tilde{\mathcal{L}}P^{k}_{{\alpha}}x^{{\alpha}}$, depending ib whether $|\alpha|$ is even or odd. Applying the assumption that $P^k_\alpha$ satisfies \eqref{coef} and noticing that $\tilde{\Lambda}=-\frac{1}{2}[(\tilde L-\tilde{\Lambda})+(\tilde{L}-\tilde{\Lambda})^{\top}]$, we see that the formula (\ref{coef}) is satisfied for $\tilde{\mathcal{L}}^{k+1}x$. Thus the lemma is proved.

\begin{table}
\caption{The odd and even power terms in $\tilde{\mathcal{L}}P^{k}_{{\alpha}}x^{{\alpha}}$ for various choice of $|\alpha |$.}
\begin{center}
{\renewcommand{\arraystretch}{2}
\begin{tabular}{c|c|c}\hline
 & Odd power term in $\tilde{\mathcal{L}}P^{k}_{{\alpha}}x^{{\alpha}}$  & Even power term in $\tilde{\mathcal{L}}P^{k}_{{\alpha}}x^{{\alpha}}$  \\
 \hline
 $|{\alpha}|$ even & $P^{k}_{{\alpha}}(\tilde{L}-\tilde{\Lambda})x\cdot\nabla x^{{\alpha}}$ & $P^{k}_{{\alpha}}\tilde{B}(x,x)\cdot\nabla x^{{\alpha}}$, $\frac{1}{2}P^{k}_{{\alpha}}\tilde{\sigma}^{2}\tilde{\Lambda}:\nabla^2 x^{{\alpha}}$\\
 \hline
 $|{\alpha}|$ odd & $P^{k}_{{\alpha}}\tilde{B}(x,x)\cdot\nabla x^{{\alpha}}$, $\frac{1}{2}P^{k}_{{\alpha}}\tilde{\sigma}^{2}\tilde{\Lambda}:\nabla^2 x^{{\alpha}}$ &$P^{k}_{{\alpha}}(\tilde{L}-\tilde{\Lambda})x\cdot\nabla x^{{\alpha}}$ \\ \hline
\end{tabular}}
\end{center}
\label{table1}
\end{table}

\end{enumerate}\qed
\end{proof}

Using this lemma, we obtain

\begin{proposition}
Parameters $\tilde{B}_i$ cannot be identified uniquely from matching $M_i=\hat{k}_x^{(i)}(t)$, where $i\geq 0$
 at $t=0$.\end{proposition}
\begin{proof}
Since at equilibrium state ${x}\sim \mathcal{N}(0,\tilde \sigma_{eq}^{2}I)$, we have that $\mathbb{E}_{\tilde{p}_{eq}}[{x}^{{\alpha}}]=0$ if $ |{\alpha}|$ is odd. Since $\hat k_{A}^{(i)}(0)=\frac{1}{\sigma_{eq}^2}\mathbb{E}_{\tilde{p}_{eq}}[(\mathcal{L}^{i}x)x^{\top}]$, it is clear that the nontrivial components of $\hat k_A^{(i)}(0)$ are contributed by only  the odd power terms of $\mathcal{\tilde L}^{i}x$. From the lemma above, it is also  clear that $\hat{k}_A^{(i)}(0)$ is always a function of $(\tilde{B}_1^2,\tilde{B}_2^2,\tilde{B}_1\tilde{B}_2,\tilde{L}-\tilde{\Lambda})$. Therefore, we cannot uniquely determine $\{\tilde B_i\}$ from matching the conditions $M_i=\hat{k}_x^{(i)}(0)$, for all $i \geq 0$.\qed
\end{proof}

The proposition above confirms that we cannot recover the parameters $\{\tilde{B}_i\}$ in \eqref{triad} uniquely by fitting only the mean response statistics. However, we claim that these parameters can be determined by fitting the response statistics of different observables, $A$. For example, consider $A(x)=(x_2x_3,x_1x_3,x_1x_2)^{\top}$. Since $A(x)$ has only even power terms, we know that $\hat{k}_{A}(0)=0$. Notice that,
\begin{equation*}
\nabla A(x)=\left(
\begin{array}{c c c} 0 & x_3 & x_2 \\ x_3 & 0 & x_1 \\ x_2 & x_1 & 0 \end{array}
\right).
\end{equation*}
Thus the only odd power term in $\tilde{\mathcal{L}}A(x)$ is $\tilde{B}(x,x)\cdot\nabla A(x)$, which is,
\begin{equation*}
\tilde{B}(x,x)\cdot\nabla A(x)=\left(\begin{array}{c c c} \tilde{B}_2 x_1x_3^2+\tilde{B}_3x_1x_2^2 \\ \tilde B_1 x_2 x_3^2+\tilde{B}_3x_1^2x_2 \\ \tilde{B}_1x_2^2x_3+\tilde{B}_2x_1^2x_3
\end{array}\right)
\end{equation*}
Thus we have
\begin{equation*}
\hat{k}_{A}'(0)=\frac{1}{\sigma_{eq}^2}\int_{\mathbb{R}^3}\tilde{B}(x,x)\cdot\nabla A(x)x^{\top}\tilde p_{eq}\td x=\frac{\tilde{\sigma}_{eq}^4}{\sigma_{eq}^2}\left(\begin{array}{c c c}
B_2+B_3 & 0 & 0\\
0 & B_1+B_3 & 0\\
0 & 0 & B_1+B_2\\
\end{array}\right)
\end{equation*}
which contains linear terms of $B_i$s. This observation suggests that in order to fully recover all the parameters in \eqref{triad}, one may have to choose multiple observables as test functions. 
\smallskip

From this example, we also notice that the term $M_1$ involves terms of order $\tilde \sigma_{eq}^2/\sigma_{eq}^2=1$ (see \eqref{M1result}) and $M_2$ involves terms of order $\sigma_{eq}^4/\sigma_{eq}^2=\sigma_{eq}^2$ (see \eqref{M2result}), where we used the fact that $\tilde \sigma_{eq}^2=\sigma_{eq}^2$ based on \eqref{M0result}. In general, both $M_{2\ell}$ and $M_{2\ell+1}$ contain terms with order $\sigma_{eq}^{2\ell}$. In practical applications, if $\sigma_{eq}\ll 1$ and  $\ell$ is large, the parameters in the order $\sigma_{eq}^{2\ell}$ terms are not identifiable. For example, the parameters $\tilde B_i^2$ in \eqref{M2result} will be harder to detect when $\sigma_{eq}\ll 1$. One way to mitigate this issue is by rescaling the equilibrium covariance of the true data to identity.

\subsection{A Langevin dynamics model}
Here, we consider a classical example in statistical mechanics: the dynamics of a particle driven by a conservative force, a damping force, and a stochastic force. In particular, we choose the force based on the Morse potential,
\begin{equation}
U(x)= U_0(a(x-x_0)), \quad U_0(x)= \epsilon(e^{-2x}-2e^{-x} + 0.01 x^2), \label{potential}
\end{equation}
where the last quadratic term acts as a retaining potential, preventing the particle to go to infinity.  This will ensure that the probability density is well defined. Here, we have introduced three parameters $(x_0,a,\epsilon).$ The parameter $x_0$ indicates a distance that corresponds to the lowest energy, the parameter $a$ controls the length scale, and $\epsilon$ has the unit of energy.

We rescale the mass to unity $m=1$, and write the dynamics as follows,
\begin{equation}\label{Lan_sys}
\begin{cases}
\dot{x}=v  \\
\dot{v}=-U'(x)-\gamma v+\sqrt{2\gamma k_{B}T} \dot{W}, \\
\end{cases}
\end{equation}
where \blue{$\dot{W}$ is a white noise and the amplitude $2\gamma k_{B}T$}, with temperature $k_{B}T$. The amplitude is chosen to satisfy the fluctuation-dissipation theorem \cite{kubo1966fluctuation}. Here we consider a one-dimensional model, $x,v\in \mathbb{R}.$ In general, the dimension depends on  the number of particles, and such models have been widely used in molecular modeling.
The dynamics in \eqref{Lan_sys} naturally introduces two additional parameters $\gamma$ and $k_B T$. Our goal is  to determine the five positive parameters $(\epsilon,\gamma,k_{B}T,a,x_0)$ from the samples of $x$ and $v$ at the equilibrium state. The equation above has the equilibrium distribution $p_{eq}(x,v)$ given by, 
\begin{equation*}
p_{eq}\propto \exp\Big[-\frac{1}{k_{B}T}\big(U(x)+\frac{1}{2}v^{2})\big)\Big].
\end{equation*}
In particular, we have $v\sim \mathcal{N}(0,k_{B}T)$ at equilibrium. 

We first look at what we can learn from the equilibrium statistics. We define 
the corresponding probability density functions in terms of both $U$ and $U_0$,
\begin{eqnarray}
\tilde{p}_{eq}^{\tilde a,\tilde x_0}(x)&:=&\frac{1}{N}\exp(-\frac{1}{k_B \tilde{T}}U_0(\tilde{a}(x-\tilde{x}_0)),\nonumber \\
\tilde{p}_{eq}^{1,0}(x)&:=&\frac{1}{N_{0}}\exp(-\frac{1}{k_B \tilde{T}}U_0(x)),\nonumber
%\end{cases}
\end{eqnarray}
where we adopt the `tilde' notation to be consistent with the formulation in Section~3. Here, the normalizing constant $N$ and $N_0$ are given by,
\begin{equation*}
N=\int_{\mathbb{R}}\exp(-\frac{1}{k_B \tilde{T}}U_0(\tilde{a}(x-\tilde{x}_0)))\td x=\frac{1}{\tilde{a}}\int_{\mathbb{R}}\exp(-\frac{1}{k_B \tilde{T}}U_0(y))\td y=\frac{1}{\tilde{a}}N_0.
\end{equation*}
As a result, the expectation with respect to $\tilde{p}_{eq}^{a,x_0}$ can often by expressed in terms of the expectation with respect to $\tilde{p}_{eq}^{1,0}$ to reveal the explicit dependence on the parameters $a, x_0$. For instance, the first moment,
\begin{equation*}
\begin{split}
\mathbb{E}_{\tilde{p}_{eq}}^{\tilde{a},\tilde{x}_0}[x]&=\frac{1}{N}\int_{\mathbb{R}}x\exp(-\frac{1}{k_B \tilde{T}}U_0(\tilde{a}(x-\tilde{x}_0))) \td x \\
&=\frac{1}{\tilde{a}N_0}\int_{\mathbb{R}}(y+\tilde{a}x_0)\exp(-\frac{1}{k_B \tilde{T}}U_0(y)) \td y \\
&=\frac{1}{\tilde{a}}\mathbb{E}_{\tilde{p}_{eq}}^{1,0}[x]+\tilde{x}_0.
\end{split}
\end{equation*}

Similarly, we get the second moment,
\begin{equation*}
\begin{split}
\mathbb{E}_{\peq}^{\tilde{a},\tilde{x}_0}[x^2]
&=\frac{1}{N}\int_{\mathbb{R}}x^2 \exp(-\frac{1}{k_B \tilde{T}}U_0(\tilde{a}(x-\tilde{x}_0))) \td x =\frac{1}{N_0}\int_{\mathbb{R}}(\frac{y}{\tilde{a}}+\tilde{x}_0)^2 \exp(-\frac{1}{k_B \tilde{T}}U_0(y)) \td y\\
&=\frac{1}{\tilde{a}^2}\left(\frac{1}{N_0}\int_{\mathbb{R}}y^2\exp(-\frac{1}{k_B \tilde{T}}U_0(y)) \td y \right)+\frac{2\tilde{x}_0}{\tilde{a}}\left(\frac{1}{N_0}\int_{\mathbb{R}}y\exp(-\frac{1}{k_B \tilde{T}}U_0(y)) \td y \right)+\tilde{x}_0^2 \\
&=\frac{1}{\tilde{a}^2}\mathbb{E}_{\tilde{p}_{eq}}^{1,0}[x^2]+\frac{2\tilde{x}_0}{\tilde{a}}\mathbb{E}^{1,0}_{\tilde{p}_{eq}}[x]+\tilde{x}_0^2.
\end{split}
\end{equation*}

Matching these moments to those computed from the data, we obtain
\BEA
\mathbb{E}_{\peq}[x] &=& \mathbb{E}_{\tilde{p}_{eq}}^{\tilde{a},\tilde{x}_0}[x]=\frac{1}{\tilde{a}}\mathbb{E}_{\tilde{p}_{eq}}^{1,0}+\tilde{x}_0, \\
\mathbb{E}_{\peq}[x^2] &=& \mathbb{E}_{\tilde{p}_{eq}}^{\tilde{a},\tilde{x}_0}[x^2] =\frac{1}{\tilde{a}^2}\mathbb{E}_{\tilde{p}_{eq}}^{1,0}[x^2]+\frac{2\tilde{x}_0}{\tilde{a}}\mathbb{E}^{1,0}_{\tilde{p}_{eq}}[x]+\tilde{x}_0^2.
\EEA
Since $\tilde a>0$, we have a unique solution
\begin{equation}\label{axo}
\begin{cases}
\tilde{a}=\sqrt{\frac{\text{Var}_{\peq}[x]}{\text{Var}^{1,0}_{\tilde{p}_{eq}}[x]}},\\
\tilde{x}_0=\mathbb{E}_{\peq}[x]-\mathbb{E}^{1,0}_{\tilde{p}eq}[x]\sqrt{\frac{\text{Var}^{1,0}_{\tilde{p}_{eq}}[x]}{\text{Var}_{\peq}[x]}},
\end{cases}
\end{equation}
where $\text{Var}_{\tilde{p}_{eq}}^{1,0}[x]$ stands for the variance of $x$ with respect to equilibrium density $\tilde{p}^{1,0}_{eq}$, which can be computed by a direct sampling. Practically, this requires solving \eqref{Lan_sys} with $\tilde{a}=1, \tilde{x}_0=0$ and the correct parameters $\gamma,\epsilon$ and $k_B T$ which we will obtain next through fitting to the essential statistics. On the other hand, the $\mathbb{E}_{\peq}[x]$, $\mbox{Var}_{\peq}[x]$ can be empirically estimated directly from the given data (which are solutions of the model using the true parameters). Based on the derivation above, we notice that fitting the one-time equilibrium statistics alone is not enough to determine the model parameters.

To estimate the other three parameters, we consider a constant external forcing $\delta f$ with $\delta\ll 1$. The corresponding perturbed system is given by,
\begin{equation}\label{psys}
\begin{cases}
\dot{x}=v  \\
\dot{v}=-U'(x)-\gamma v+\delta f+\sqrt{2\gamma k_BT} \dot{W}. \\
\end{cases}
\end{equation}

%Due to uncertainty principle, 
We select the observable $A=(0,v)^\top$, and take $c(x)=(0,1)$ in the FDT formula. 
\comment{To begin with the calculation of the essential statistics, we have
\begin{equation*}
%\begin{cases}
%\frac{\partial}{\partial x}\peq=-p_{eq}\frac{U'(x)}{k_{B}T}, \\
\frac{\partial}{\partial v}\peq=- \frac{v}{k_B T}p_{eq}.
%\end{cases}
\end{equation*}}
We will work with the (2,2) entry of the response operator,  given by,
\begin{equation}
k_{A}(t) = \frac{1}{k_B T}\int_{\mathbb{R}^2}\int_{\mathbb{R}^2} 
 v v_0p(x,v,t|x_0,v_0,0) \td x\td v\td x_0 \td v_0,\nonumber
\end{equation}
which can only be computed from the solutions of \eqref{Lan_sys} since the true parameters are not known. \blue{We denote $M_i=k^{(i)}_A(0)$, the $i^{th}$ derivative of $k_A(t)$ at point zero. We now discuss $M_0$ and higher order derivatives, $M_1$, $M_2$ and $M_3$.
}

\subsubsection{Determining the parameters from the essential statistics at $t=0$}

To make a consistent notation with our formulation in Section~3, we now use `tilde' to indicate the parameters that we are interested to estimate, namely $\{\tilde{\gamma},\tilde{\epsilon},k_B\tilde{T}\}$ where the model is exactly the Langevin model in \eqref{Lan_sys}. We define $ \tilde{p}(x,v,t|y,u,0)$ to be the solution of the Fokker-Planck equation \cite{risken1984fokker},
\BEA
\frac{\partial}{\partial t} \tilde{p} = \mathcal{\tilde{L}}^* \tilde{p}, \quad \tilde{p}(x,v,0) = \delta(x-y,v-u),\nonumber
\EEA
where the operator $\mathcal{\tilde L}^*$ is the forward operator of \eqref{triad}. The generator,  
\BEA
\tilde{\mathcal{L}} =  v\frac{\partial}{\partial x} + (-U'(x)-\tilde{\gamma}v) \frac{\partial}{\partial v} + \tilde\gamma k_B\tilde{T} \frac{\partial^2}{\partial v^2}\label{genapproxSDE2}
\EEA
is the adjoint of $\mathcal{\tilde L}^*$\comment{ with respect to $L^2(\mathbb{R}^2)$} and we denote the equilibrium density as $\tilde{p}_{eq}$.

Using the definition in \eqref{hatkA}, one can deduce that,
\begin{equation}
\hat{k}_{A}(t) = \frac{1}{k_B T}\int_{\mathbb{R}^2}\int_{\mathbb{R}^2} 
v u \tilde{p}(x,v,t|y,u,0) \td x\td v\td y \td u.\label{kt}
\end{equation}
Setting $t=0$ we have,
\begin{equation*}
\hat{k}_{A}(0) = \frac{1}{k_B T} \int_{\mathbb{R}^2} 
 v^2 \tilde{p}_{eq}(x,v) \td x\td v 
=\frac{k_B\tilde{T}}{k_B T}.
\end{equation*}
Since $M_0=1 = \hat{k}_{A}(0) = \frac{k_B\tilde{T}}{k_B T}$, it is clear that fitting the equilibrium variance reveals the temperature, $\tilde{T}= T$.

We now turn to the first order derivative,
\BEA
%\begin{equation}\label{kprime}
%\begin{split}
\hat{k}_{A}'(t) &=&  \frac{1}{k_B T}\int_{\mathbb{R}^2}\int_{\mathbb{R}^2} \mathcal{\tilde L}v\,u \tilde p(x,v,t|y,u,0) \td x\td v\td y \td u \nonumber\\
&=&  \frac{1}{k_B T}\int_{\mathbb{R}^2}\int_{\mathbb{R}^2}
(-U'(x)-\tilde \gamma v)u \tilde{p}(x,v,t|y,u,0) \td x\td v\td y \td u.\label{kprime}
%\end{split}
%\end{equation}
\EEA
Setting $t=0$ we have
\begin{equation}
\hat{k}_{A}'(0)= \frac{1}{k_B T}\int_{\mathbb{R}^2}
(-U'(x)-\tilde \gamma v)v
 \tilde{p}_{eq}(x,v)\td x\td v=  -\tilde \gamma \frac{k_B\tilde{T}}{k_BT}.\nonumber
\end{equation}
which means that the constraint $M_1= \hat{k}'_A(0)= -\tilde \gamma$, accounting $\tilde{T}=T$ from matching $M_0$. We should point out that for this problem, one can in fact compute $M_1 = -\gamma$ analytically and thus the true parameter is obtained, $\tilde{\gamma}=\gamma$. Of course, in general when the underlying model is unknown, one can't compute $M_1$ analytically and  and has to resort to  numerical approximations such as a finite difference scheme or a rational approximation as we shall see below.

To estimate $\tilde{\epsilon}$, we take higher order derivatives,
\begin{equation*}
\begin{split}
\hat{k}_{A}''(t)&=  \frac{1}{k_B T}\int_{\mathbb{R}^2}\int_{\mathbb{R}^2} \mathcal{\tilde L}
(-U'(x)-\tilde\gamma v) u \tilde{p}(x,v,t|y,u,0) \td x\td v\td y \td u \\
&= \frac{1}{k_B T}\int_{\mathbb{R}^2}\int_{\mathbb{R}^2}
(-vU''(x)+\tilde \gamma U'(x)+\tilde \gamma^{2} v) u
 \tilde{p}(x,v,t|y,u,0) \td x\td v\td y \td u.
\end{split}
\end{equation*}
Setting $t=0$ we have
\begin{equation}
\hat{k}''_{A}(0)=  \frac{1}{k_B T}\int_{\mathbb{R}^2}
(-vU''(x)+\tilde \gamma U'(x)+\tilde \gamma^{2} v) v
 \tilde{p}_{eq} \td x\td v= \tilde\gamma^{2}-\mathbb{E}_{\tilde{p}_{eq}}[U''(x)].\label{k2A}
\end{equation}
Notice that the parameter $\tilde \epsilon$ is hidden in the potential $U$ as defined in \eqref{potential} and the expectation is with respect to the equilibrium density, $\tilde{p}_{eq}$, which also depends on the parameters $\tilde \epsilon$. 

\comment{
For the third order derivative, we find that (see Appendix B),
\begin{equation*}
\begin{split}
\hat{k}^{(3)}_{A}(0)= 2\tilde\gamma\mathbb{E}_{eq}[U''(x)]-\tilde\gamma^3. 
\end{split}
\end{equation*}
Interestingly, it is not independent of $\hat{k}''_{A}(0)$.
}

\blue{To summarize, a natural parameter estimation method would consist of matching the following conditions:  
\BEA
M_0 &=& \hat{k}_A(0) = \frac{k_B\tilde{T}}{k_B T}\nonumber\\
M_1 &=& \hat{k}'_A(0) = -\tilde{\gamma}\frac{k_B\tilde{T}}{k_B T}\label{M0-M2}\\
M_2 &=& \hat{k}''_A(0) = \tilde\gamma^{2}-\mathbb{E}_{\tilde{p}_{eq}}[U''(x)] \nonumber
\EEA
where the first two equations give estimates to $\tilde{T}$ and $\tilde{\gamma}$, whereas $\tilde{\epsilon}, \tilde{a},$ and $\tilde{x}_0$ can be estimated by solving the third equation together with constraints \eqref{axo}. Assuming that $\mathbb{E}_{\tilde{p}_{eq}}[U'']$, $\tilde{a}$ and $\tilde{x}_0$ are smooth functions of $\tilde{\epsilon}$, we can solve 
\BEA
M_2=\hat{k}''_A(0;\tilde{\epsilon},\tilde{a}(\tilde\epsilon),\tilde{x}_0(\tilde\epsilon)),\nonumber
\EEA
for $\tilde{\epsilon}$, where we have implicitly inserted the constraints in \eqref{axo} to the third equation in \eqref{M0-M2}. In general, we expect that $\hat{k}^{(i)}(0)$ do not have explicit expressions and thus the proposed algorithm in this paper should be carried numerically. Therefore an important scientific problem will be to design an efficient numerical method to sample $\tilde{p}_{eq}$ which depends on the parameters that are to be estimated. }

In principle, to solve the system of equations in \eqref{M0-M2} along with the constraints \eqref{axo}, we need to estimate $M_0, M_1, M_2$ from the data and simultaneously check whether these statistics are sufficient to estimate the FDT response operator $k_A(t)$. We discuss the practical aspects in the next section.

\subsubsection{Practical challenges in approximating the response operator and the higher-order essential statistics}

{
The goal of this section is to show that it is indeed difficult to approximate the long term response operator using just the essential statistics $M_i$. In fact, it is already difficult to estimate the higher order derivatives such as $M_2$ and $M_3$. Before we give a remedy, let us illustrate this issue with a few numerical results.

Since the linear response operator $k_A(t)$ in this case is a scalar function, the order-1 approximation corresponds to $k_A(t)\approx g_1(t) = e^{t\beta_1}\alpha_1$. The general order-$m$ method $k_A(t)$ is approximated by $g_m(t)$ in \eqref{eq: ansatz}. We conducted two numerical experiments, one with $\gamma=0.5$ and a second one with $\gamma=0.1$. The correlation function $k_A(t)$ is computed using $8\times10^6$ samples obtained by subsampling every $10$ steps of the solutions up to $20000$ model unit time with truncating the transient time. For $\gamma=0.5$, the linear response kernel shows rapid decay, indicating an over-damped behavior, as shown in Figure \ref{fig: over-damp}, where we also show the order-1, 2, and 3 approximations of the response kernel using the short-time essential statistics. Here, by order-$m$ approximation, we mean that the coefficients in $g_m$ are determined by $M_i, i=0,\ldots, 2m-1$, where $M_i$s are computed explicitly (see Appendix B for the higher order terms). Here, we choose to use the true $M_i$ to rule out the potential effects of numerical error. It is clear that all three approximations show good agreement only at short times. Beyond $t>1$, the estimates are not accurate at all. The inaccuracies in the response function approximation are even more apparent in the case of weaker damping $\gamma=0.1$, in which case, the response kernel exhibits an oscillatory pattern, indicating that there is a nontrivial intermediate scale dynamics. The main observation here is that when only the short-time essential statistics $M_i = k^{(i)}_A(0)$ is used to determine the parameters in $g_m$, the prediction of response function has very poor accuracy. In particular, while the higher order approximations offer a slightly better approximation of the response function near zero, the overall accuracy is very poor.

}
\begin{figure}[htbp]
\begin{center}
\includegraphics[scale=0.4]{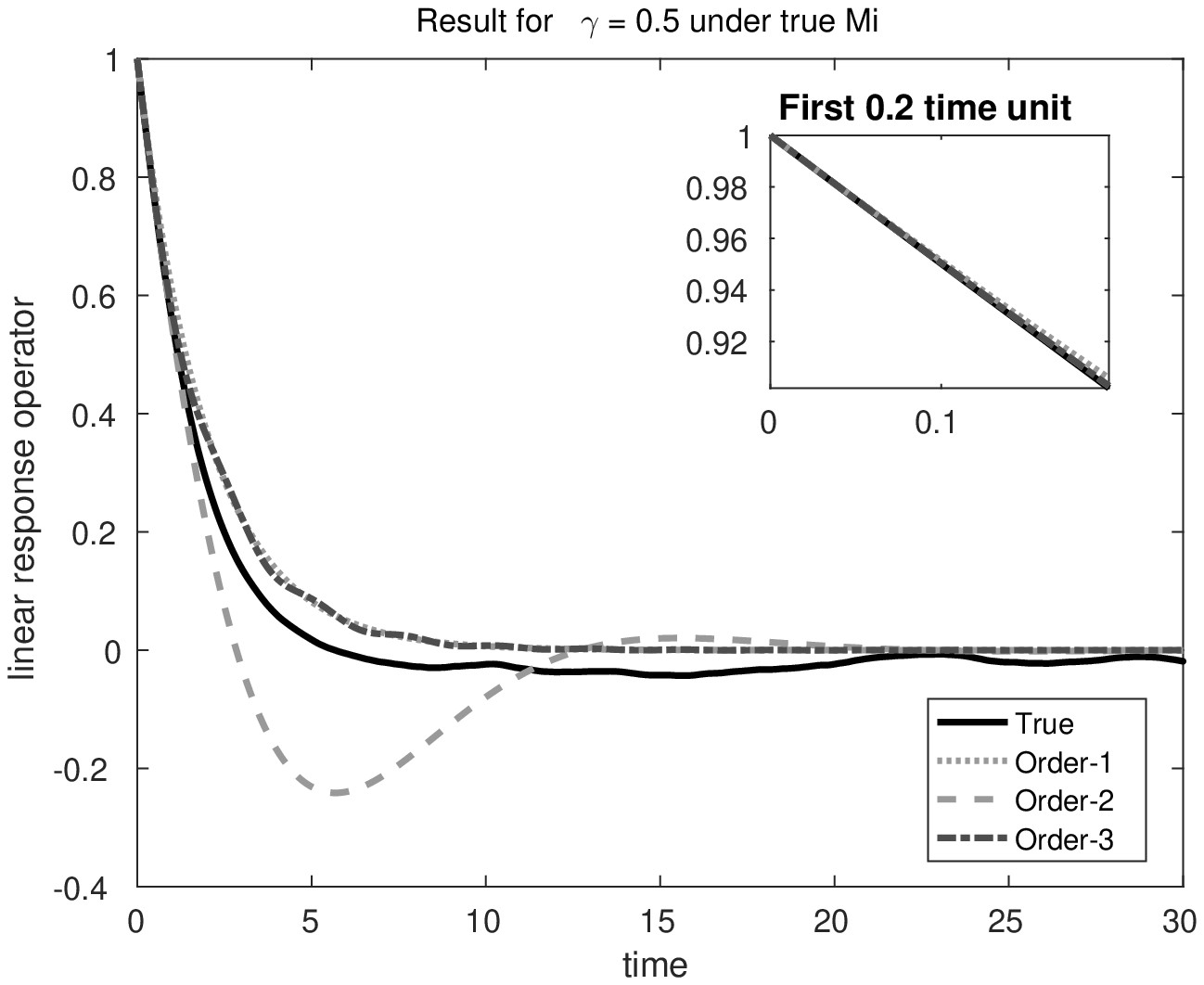}
\includegraphics[scale=0.4]{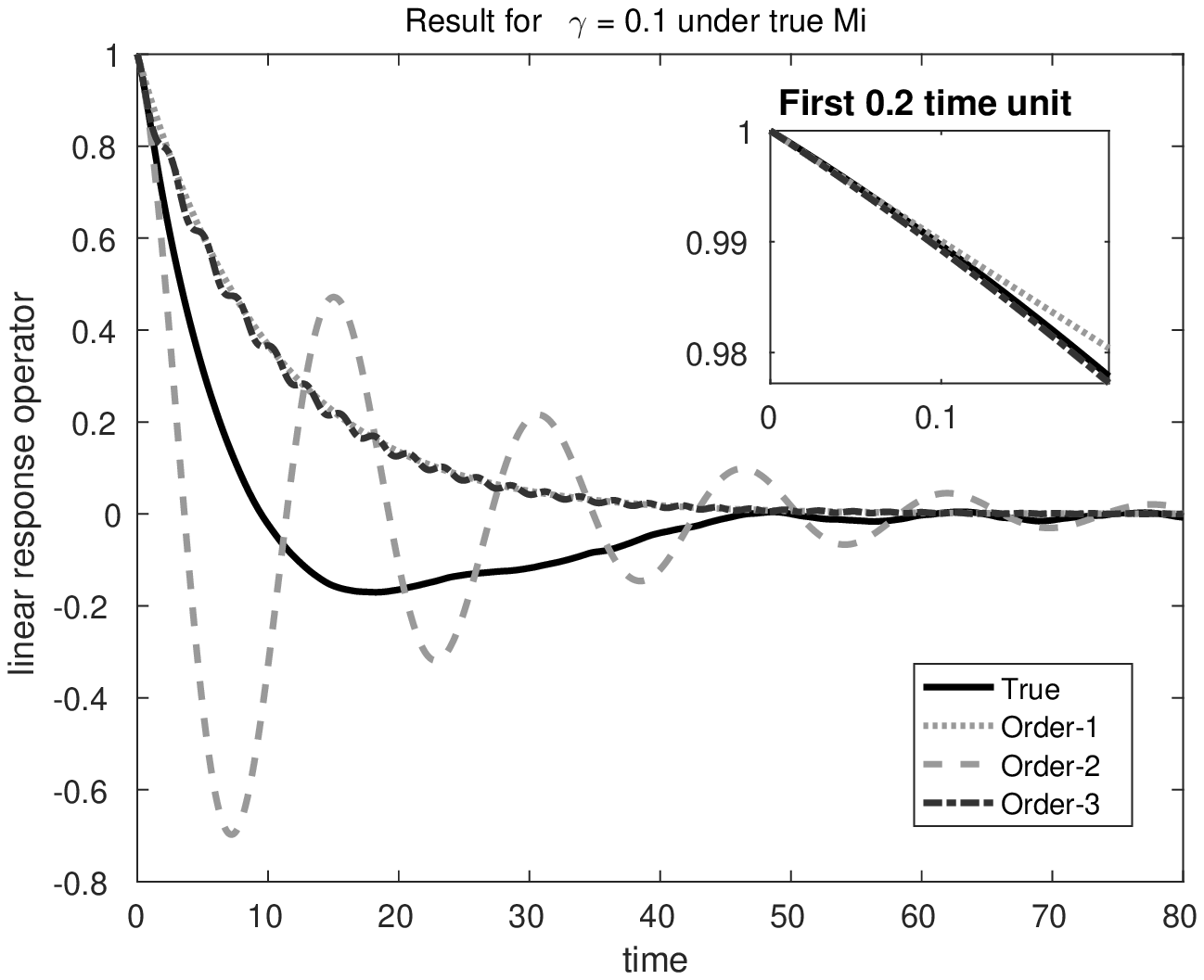}
\caption{The response functions for $\gamma=0.5$ (left) and $\gamma=0.1$ (right). The exact response function (solid line) is obtained from a direct numerical simulation, and the approximate response functions are obtained from the order-$m$ rational approximations  \eqref{eq: ansatz}.}
%\caption{\blue{The response functions. The exact response function (solid line) is obtained from the data. The approximate response functions are obtained from the order-1 (dotted line) and order-2 (dashed line) rational approximations with true values of $M_i$.}}
\label{fig: over-damp}
\end{center}
\end{figure}

\blue{This result clearly suggests that the essential statistics at $t=0$ are not sufficient in approximating the response function. A simple remedy is to include the essential statistics from intermediate time scales, that is, ${k}_A^{(i)}(t_j)$, where $t_j>0$. This means that we need to estimate the derivatives at these locations. We propose to approximate these derivatives with $g^{(i)}_m(t_j)$, where the parameters in $g_m(t)$ in \eqref{eq: ansatz} are obtained 
by a nonlinear least squares fitting. More specifically, we estimate the parameters in $g_m(t)$ by solving,
\BEA
\min_{\substack{\alpha_1,\ldots,\alpha_m \\ \beta_1,\ldots,\beta_m}} \sum_{i=1}^{n} \Big(k_{A}(t_{i})-g_{m}(t_{i};\alpha_1,\ldots,\alpha_m,\beta_1,\ldots,\beta_m)\Big)^{2}, \label{nls}
\EEA
where is a nonlinear least-square problem. Then we use $g^{(i)}_m(t)$ to approximate $k_A^{(i)}(t)$. With this fitting, we avoid solving a system of algebraic equations that maps the coefficients \eqref{eq: ansatz} to the set of essential statistics $M_0$ and $M_1$ at the selected time steps. As an example, we pick $n=84$ points in $[0,60]$, denser near zero. We observe from the results in  Figure~\ref{fig: nonlinear-fit} that for $\gamma=0.5$, both the order-1 and 2 approximations are relatively accurate when $t<5$; here, the second order approximation is more accurate. For $\gamma=0.1$, the order-4 approximation provides an excellent fit to the response function. In Figure~\ref{fig: derv}, we show the remarkable agreement between the corresponding $g'_m(t)$ and the $k'_A(t)$ which can be computed explicitly with the formula in \eqref{kprime}. These numerical results suggest that the ansatz $g_m(t)$ in \eqref{eq: ansatz} is appropriate for estimating $k_A(t)$, but the coefficients have to be determined based on $k_A(t_j)$ at intermediate time scales.}
\smallskip
\begin{figure}[htbp]
\begin{center}
\includegraphics[scale=0.4]{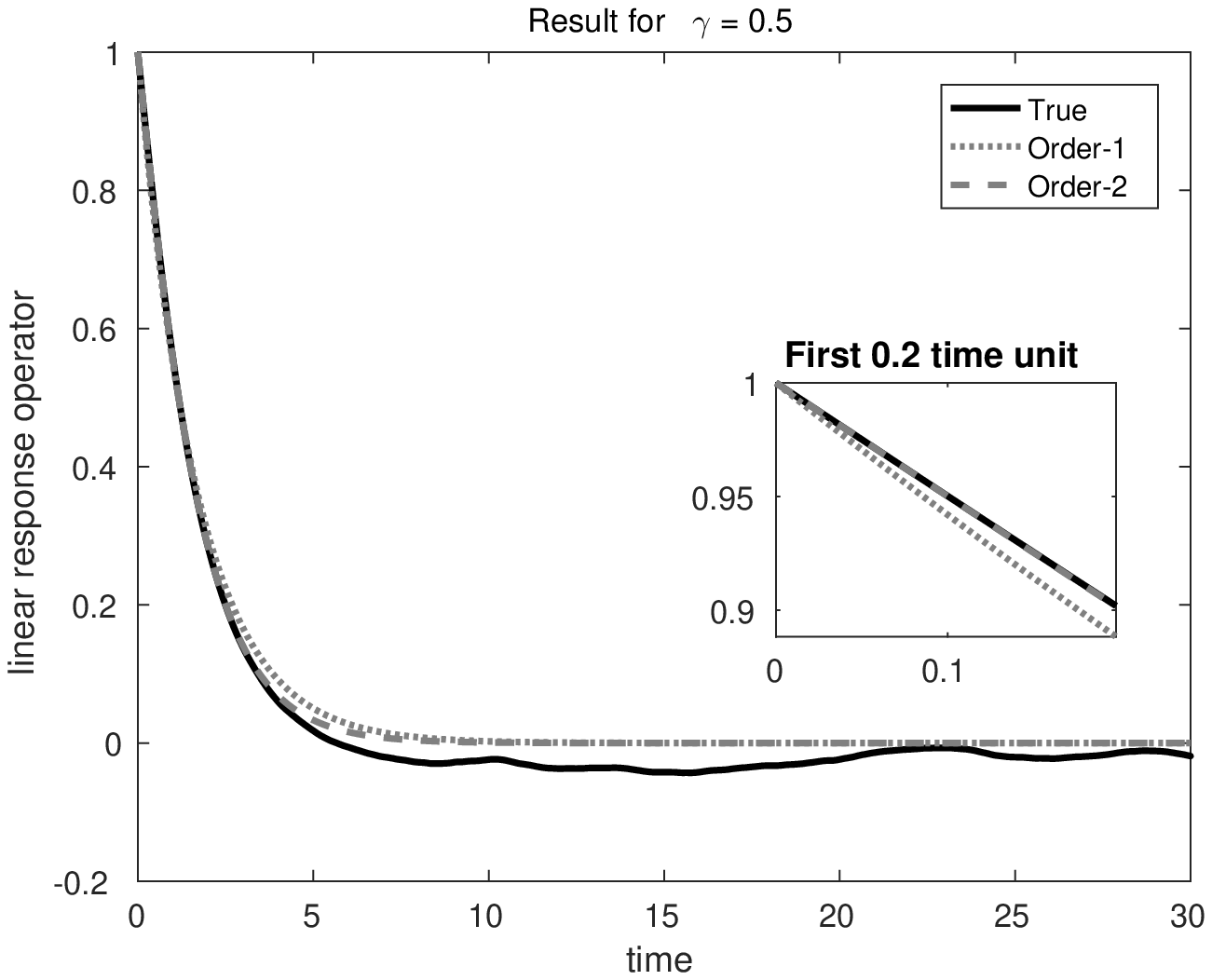}
\includegraphics[scale=0.4]{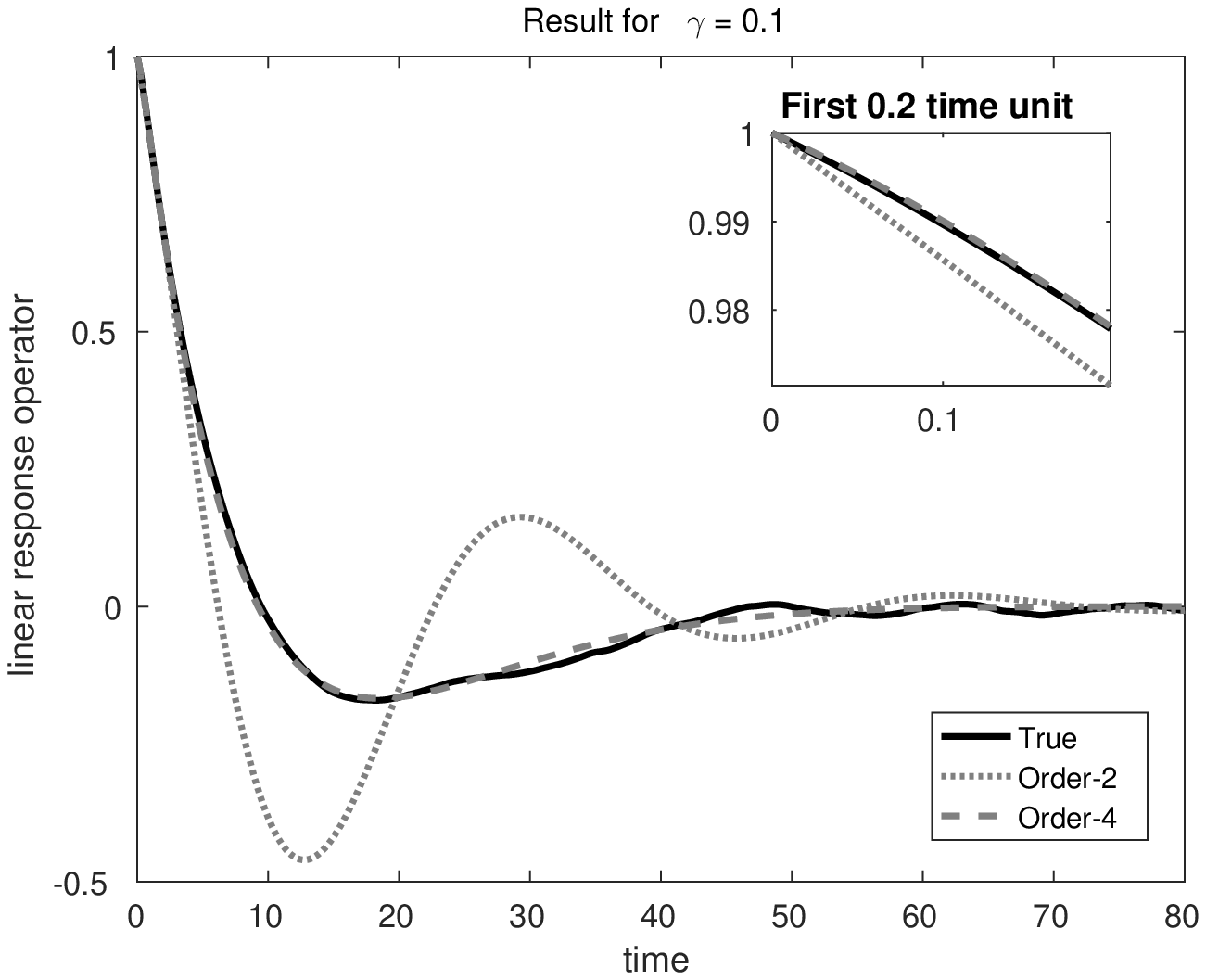}
\caption{The response functions for $\gamma=0.5$ (left) and $\gamma=0.1$ (right). The exact response function (solid line) is obtained from a direct numerical simulation, and the approximate response functions are obtained from the order-$m$ rational approximations with coefficients determined via nonlinear least-square fitting \eqref{nls}
.}
\label{fig: nonlinear-fit}
\end{center}
\end{figure}

\begin{figure}[thbp]
\begin{center}
\includegraphics[scale=0.4]{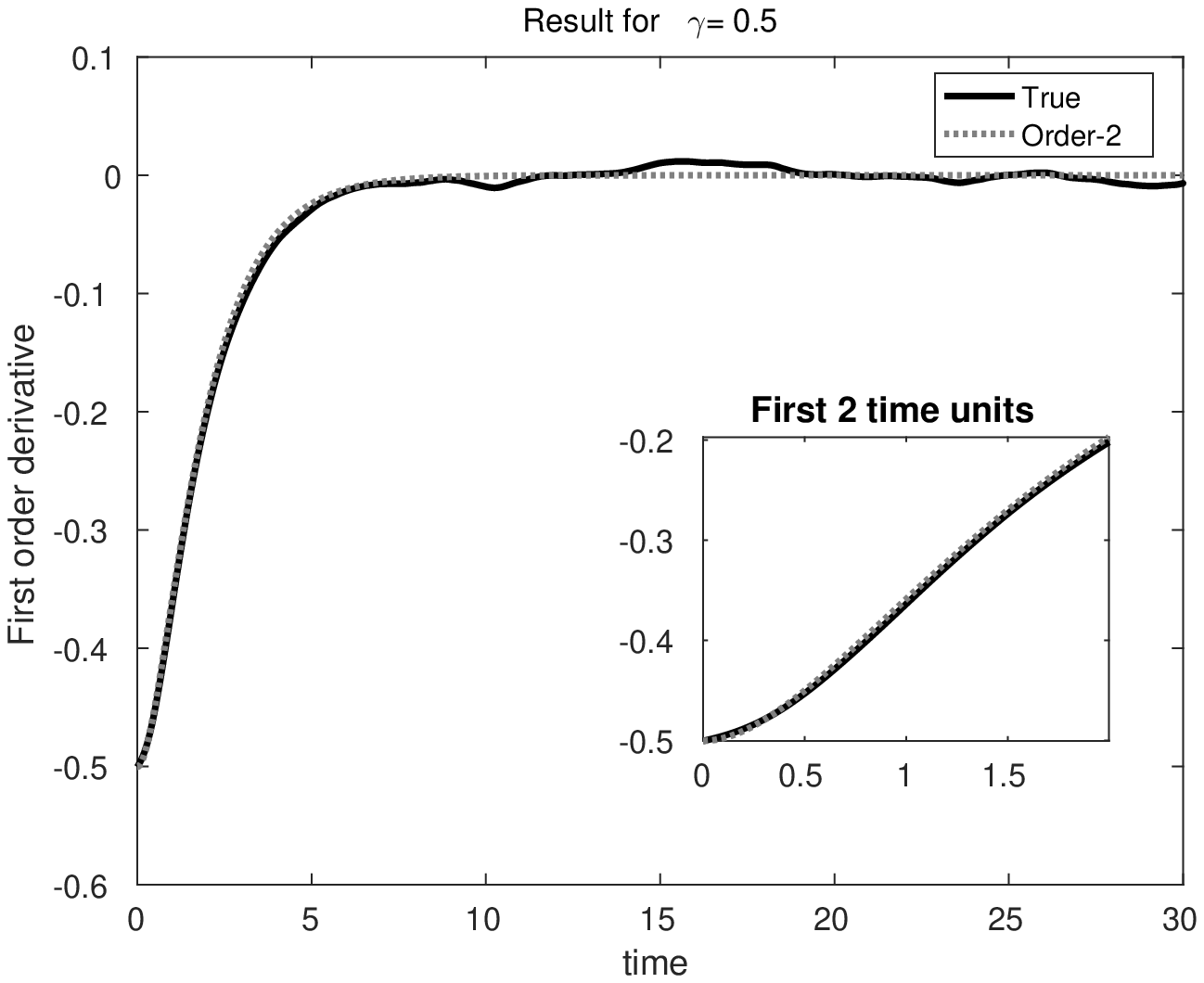}
\includegraphics[scale=0.4]{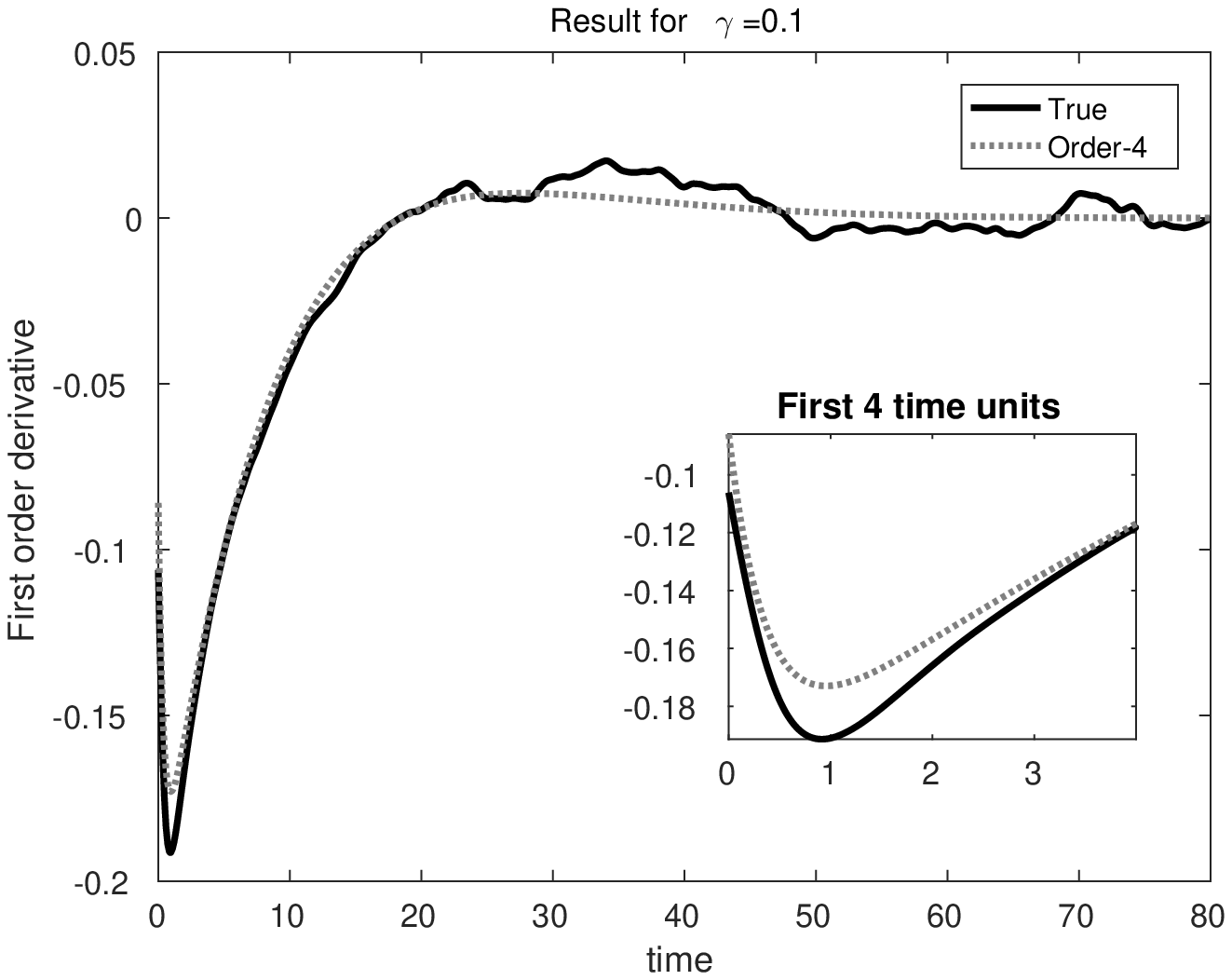}
\caption{The comparison of the first order derivative of response functions for $\gamma=0.5$ (left) and $\gamma=0.1$ (right).}
\label{fig: derv}
\end{center}
\end{figure}

\blue{Back to our parameter estimation strategy in \eqref{M0-M2}, we need to check the accuracy of $M_2$ based on these nonlinear least-square fitting. In general, when the underlying dynamics are unknown, the higher order derivatives can only be estimated from data, which are often subject to measurement or numerical error. Therefore, it is unlikely that the higher order essential statistics can be computed with good accuracy. To elaborate on this point, we check the accuracy of $M_1, M_2, M_3$ of our Langevin example where the data is generated from the weak trapezoidal method \cite{weak-trapezoildal-method}. The estimated $\{M_i\}$ are listed in Table~\ref{table2} below. As suspected, while the estimates for $M_1$ are reasonably accurate (as shown in Figure~\ref{fig: derv}), we can see that the higher order essential statistics, $M_2$ and $M_3$, are difficult to estimate. This issue might be related to the sampling and numerical error in the data used to compute the linear response operator $k_A(t)$. 

\smallskip
On the other hand, if we look at the essential statistics of the lower order, $k'(t_j)$, at {\it different} points, such as $t_j=5/2$ and $t_j=5$, they can be reasonably estimated with $g'_m(t_j)$, where $g_m$ is the rational approximation with parameters specified through the nonlinear least-squares fitting in \eqref{nls}. This suggests that in practice we may consider the fitting strategy in \eqref{mainfitting2}. In particular, consider the essential statistics $k_A(t_j)$ or $k'_A(t_j)$ where $t_j\neq 0$ in placed of the equation involving $M_2$ in \eqref{M0-M2}. For example, we can consider solving the first two equations in \eqref{M0-M2} for $\tilde{\gamma}, \tilde{T}$ and
\BEA
k_A(t_j) &=& \hat{k}_A(t_j) = \frac{1}{k_B T}\int_{\mathbb{R}^2}\int_{\mathbb{R}^2} 
v u \tilde{p}(x,v,t_j|y,u,0) \td x\td v\td y \td u, \label{katj}
\EEA
for $\tilde{\epsilon}$, where the last equality is from \eqref{kt}. Here, $v=v(\tilde{\epsilon},\tilde{\gamma},\tilde{T})$ on the right hand side of the third equation above, where we implicitly used the constraints in \eqref{axo} which imply the dependence $\tilde{x}_0(\epsilon)$ and $\tilde{a}(\tilde{\epsilon}$). Alternative to \eqref{katj}, we can solve,
\BEA
g'_m(t_j) \approx k'_A(t_j) = \hat{k}'_A(t_j) = \frac{1}{k_B T}\int_{\mathbb{R}^2}\int_{\mathbb{R}^2}
(-U'(x)-\tilde \gamma v)u \tilde{p}(x,v,t_j|y,u,0) \td x\td v\td y \td u \label{kaptj}
\EEA
for $\tilde{\epsilon}$, where the last equality is from \eqref{kprime}.

Here, an appropriate numerical method for solving either \eqref{katj} or \eqref{kaptj} is needed. In particular, the computational cost will be in sampling $\tilde{p}$ which depends on the parameters to be estimated. We will deal with this issue in our future report.}

\begin{table}
\caption{Comparison of the essential statistics $M_i$ and $g^{(i)}_m(0)$ for $i=1, 2, 3$, where $g_m$ is the rational approximation with parameters specified through the nonlinear least-squares fitting in \eqref{nls}.}
\begin{center}
{\renewcommand{\arraystretch}{2}
\begin{tabular}{c|cc|cc}
\hline
\multicolumn{3}{c}{$\qquad\gamma=0.5$} & \multicolumn{2}{c}{$\gamma=0.1$}\\
\hline 
 & True Value & Order-$2$ & True Value & Order-$4$ \\
\hline
$M_1$ & -0.5000 & -0.5002 & -0.1000  & -0.0859  \\
\hline
$M_2$ & 0.0861 & -0.0003 & -0.1539  & -0.2985  \\
\hline
$M_3$ & 0.0389 & 0.5851 & 0.0318 & 0.8714  \\
\hline
\end{tabular}}
\end{center}
\label{table2}
\end{table}

\begin{table}
\caption{Comparison of the essential statistics $k'(t_j)$ and $g'_m(t_j)$ at $t_j=\frac52, 5$ where $g_m$ is the rational approximation  \eqref{eq: ansatz} with parameters specified through the nonlinear least-squares fitting in \eqref{nls}.}
\begin{center}
{\renewcommand{\arraystretch}{2}
\begin{tabular}{c|cc|cc}\hline
 \multicolumn{3}{c}{$\qquad\quad\gamma=0.5$} & \multicolumn{2}{c}{$\gamma=0.1$}\\  \hline
& True Value & Order-2  & True Value & Order-4 \\ \hline
$k'_{A}(5/2)$ & -0.14792 & -0.14107 & -0.15220 & -0.14626 \\ 
$k'_{A}(5)$ &-0.02874 & -0.02403 &  -0.10052 & -0.09988 \\ \hline
\end{tabular}}
\end{center}
\label{table3}
\end{table}

\section{Concluding discussion}

{ This paper presents a new framework for model parameterization using linear response statistics. The proposed approach is motivated by the observation that in many applications, some of the model parameters cannot be determined solely based on the equilibrium statistics. In fact, we provided two examples in which some of the parameters do not even appear in the equilibrium probability density. We proposed to identify the parameters by using dynamic response properties. In essence, it is very similar to the impulse/response experimental techniques, where one introduces an external field and then infers the underlying structure using the observed response.   

For this purpose, we  made use of the fluctuation-dissipation theory, and we introduced the concept of essential statistics, a finite set of quantities that is capable of  representing the overall response function. 
With a simplified turbulence model, we demonstrated how the model parameters can be determined analytically from the essential statistics at one point $t=0$, and equally importantly, why these essential statistics are capable of representing the response function. From this example, we also learned that we may need to consider various observable functions to estimate all the of parameters in the model. 

For the second example, which is a mechanical model, the calculations are semi-analytical in the sense that some of the computations must be done numerically. While it is true that we can identify the model parameter by fitting to the essential statistics at $t=0$, we learned that these statistics are not sufficient to represent the response operator at the long time. As a result, essential statistics from intermediate time scales must be included. Practically, this also requires estimating the derivatives of the response functions at different times. We showed that we can approximate the essential statistics accurately at least up to the order-1 derivative by a nonlinear least-square fitting to the rational function in \eqref{eq: ansatz}. For accurate estimation of the higher order derivatives,  we suspect that the data needs to be noiseless and generated with high order numerical integration schemes, which is not feasible in general. As a remedy to this practical issue, we propose to match only up to the first-order derivatives at different points as in \eqref{mainfitting2} in placed of the higher-order derivatives in \eqref{mainfitting}. 

Notice that in deriving the expression of the essential statistics as functions of the model parameters, $\hat{k}^{(i)}(t)$, all we need is the generator $\tilde{\mathcal{L}}$ of the model \eqref{approxSDE} that we want to parameterize. This implies that our formulation can be applied even if we have no access to the underlying dynamics as in many real world problems. While the goal of this paper is to show consistency, that is, we consider only $\mathcal{L}(\theta)=\mathcal{\tilde{L}}(\theta)$, we will consider the case with model error, $\mathcal{L}(\theta)\neq\mathcal{\tilde{L}}(\theta)$, in our future study.
 
Finally, the most important issue that hasn't been covered in this paper is how to devise efficient numerical methods to solve the system of equations in \eqref{mainfitting2}. The difficulty is that evaluating these equations involves sampling of either $\tilde{p}_{eq}(\tilde{\theta})$ or $\tilde{p}(x,v,t|y,u,0;\tilde{\theta})$, which depends on the parameters to be determined. This means that if an iterative type scheme is used to solve the nonlinear equations in \eqref{mainfitting2}, at each step we need to solve the approximate model in \eqref{approxSDE} with new parameters. Clever numerical algorithms with fast convergence or those without solving the model in \eqref{approxSDE} repeatedly is needed. This
important scientific computational challenge will be pursued in our future study. 

\section*{Acknowledgements}
This research was supported by the NSF grant DMS-1619661. HZ was partially supported as a GRA under the NSF grant DMS-1317919. 

\section*{Appendix A: The computation of $M_2$ for the simplified turbulence model}

%\appendix
%\subsection{The computation of $M_2$ for the simplified turbulence model}
Here, we provide the computational detail for obtaining \eqref{M2result}. In particular,
\begin{equation}
\hat{k}''_{x}(0)=\frac{1}{\sigma_{eq}^{2}}\int_{\mathbb{R}^{3}}\tilde{\mathcal{L}}(\tilde{B}(x,x)+\tilde{L}x-\tilde{\Lambda}x)x^{\top} \tilde{p}_{eq}(x)\td x. \nonumber
\end{equation}
where $\tilde{B}(x,x)=(\tilde{B}_1x_2x_3,\tilde{B}_2x_1x_3,\tilde{B}_3 x_1x_2)$ and $\tilde{L}$, $\tilde{\Lambda}$ are $3\times 3$ matrices.
We introduce the following notations.

\begin{enumerate}
    \item Set $\tilde{D}(x)=\tilde{B}(x,x)+\tilde{L}x-\tilde{\Lambda}x$ which is the deterministic part of the system, and $\tilde{D}_{i}$ denotes the $i$th component of $\tilde{D}(x)$.
    \item Set $\tilde{C}=(\tilde{L}-\tilde{\Lambda})$, then we have $\tilde{D}(x)=\tilde{B}(x,x)+\tilde{C}x$, and $\tilde{C}_i$ denotes the $i${th} row of the matrix, which is a row vector.
\end{enumerate}

With these notations, we have the Jacobian matrix of $\tilde{D}(x)$ 
\begin{equation}
\nabla \tilde{D}(x)=\nabla \tilde{B}(x,x)+\nabla (\tilde{C}x)=\left(\begin{array}{ccc}
    0 & \tilde{B}_1 x_3 & \tilde{B}_1 x_2  \\
    \tilde{B}_2 x_3 & 0 & \tilde{B}_2 x_1  \\
    \tilde{B}_3 x_2 & \tilde{B}_3 x_1 & 0
\end{array}\right)+\tilde{C},\nonumber
\end{equation}
and the generator $\tilde{\mathcal{L}}$ acting on $f$ becomes,
\begin{equation}
\tilde{\mathcal{L}}f =\sum_{i=1}^{3}\tilde{D}_{i}\frac{\partial f}{\partial x_{i}}+\frac{\tilde{\sigma}^{2}}{2}\sum_{i,j=1}^{3}(\tilde{\Lambda})_{i,j}\frac{\partial^{2}f}{\partial x_{i}\partial x_{j}}.\nonumber
\end{equation}

In our case, $f=\tilde{D}(x)$. Since $\tilde{p}_{eq}\sim \mathcal{N}(0,\tilde{\sigma}_{eq}^{2})$, with $\tilde{\sigma}_{eq}^{2}=\frac{\tilde{\sigma}^{2}}{2}$, only the odd power terms in $\tilde{\mathcal{L}}\tilde{D}(x)$ contribute to $M_2$. Notice that the second order partial derivatives of $\tilde{D}(x)$ are constant, and thus we have,
\begin{equation}
\hat{k}''_{x}(0)=\frac{1}{\sigma_{eq}^{2}}\int_{\mathbb{R}^{3}}(\tilde{\mathcal{L}}\tilde{D}(x))x^{\top}\tilde{p}_{eq}\td x= \frac{1}{\sigma_{eq}^{2}}\int_{\mathbb{R}^{3}}\left(\sum_{i=1}^{3}\tilde{D}_{i}\frac{\partial \tilde{D}}{\partial x_{i}}\right)x^{\top}\tilde{p}_{eq}\td x. \nonumber
\end{equation}
For example, if $i=1$,
\begin{equation}
\begin{split}
\tilde{D}_{1}\frac{\partial \tilde{D}}{\partial x_1} &=(\tilde{B}_1x_2x_3+\tilde{C}_1 x)[(0,\tilde{B}_2x_3,\tilde{B}_3x_2)^{\top}+\tilde{C}_1^{\top}]\\
&=(0,\tilde{B}_1\tilde{B}_2x_2x_3^2,\tilde{B}_1\tilde{B}_3x_2^2x_3)^{\top}+\tilde{B}_1x_2x_3\tilde{C}_1^{\top}+\tilde{C}_1x(0,\tilde{B}_2x_3,\tilde{B}_3x_2)^{\top}+\tilde{C}_1x\tilde{C}_1^{\top}.
\end{split}
\nonumber
\end{equation}
Notice that both $\tilde{B}_1x_2x_3\tilde{C}_1^{\top}$ and $\tilde{C}_1x(0,\tilde{B}_2x_3,\tilde{B}_3x_2)^{\top}$ are even power terms which means that they do not affect the value of $M_2$. That is,
\begin{equation*}
\begin{split}
\int_{\mathbb{R}^{3}}\tilde{D}_{1}\frac{\partial \tilde{D}}{\partial x_1}x^{\top}\tilde{p}_{eq}\td x 
&= \int_{\mathbb{R}^{3}} \left(\begin{array}{c} 0 \\ \tilde{B}_1\tilde{B}_2 x_2x_3^2 \\ \tilde{B}_1\tilde{B}_3 x_2^2x_3 \end{array} \right)x^{\top}\tilde{p}_{eq}+\tilde{C}_{1}^{\top}\tilde{C}_{1}xx^{\top}\tilde{p}_{eq}\td x \\
&= \int_{\mathbb{R}^{3}} \left(\begin{array}{c c c} 0 & 0 & 0 \\ \tilde{B}_1\tilde{B}_2 x_1x_2x_3^2 & \tilde{B}_1\tilde{B}_2x_2^2x_3^2 & \tilde{B}_1\tilde{B}_2 x_2x_3^3 \\ \tilde{B}_1\tilde{B}_3 x_1x_2^2x_3 & \tilde{B}_1\tilde{B}_3 x_2^3x_3 & \tilde{B}_1\tilde{B}_3 x_2^2x_3^2 \end{array} \right)\tilde{p}_{eq} \td x+\tilde{C}_{1}^{\top}\tilde{C}_{1}\int_{\mathbb{R}^{3}}xx^{\top}\tilde{p}_{eq}\td x \\
& = \tilde{\sigma}_{eq}^4 \left(\begin{array}{c c c} 0 & 0 & 0 \\ 0 & \tilde{B}_1\tilde{B}_2 & 0 \\ 0 & 0 & \tilde{B}_1\tilde{B}_3 \end{array} \right)+\tilde{\sigma}_{eq}^2\tilde{C}_{1}^{\top}\tilde{C}_{1}.
\end{split}
\end{equation*}
Similarly, for $i=2$ and $i=3$, we have
\begin{equation*}
\begin{split}
\int_{\mathbb{R}^{3}}\tilde{D}_{2}\frac{\partial \tilde{D}}{\partial x_2}x^{\top}\tilde{p}_{eq}\td x =  \tilde{\sigma}_{eq}^4 \left(\begin{array}{c c c} \tilde{B}_1\tilde{B}_2 & 0 & 0 \\ 0 & 0 & 0 \\ 0 & 0 & \tilde{B}_2\tilde{B}_3 \end{array} \right)+\tilde{\sigma}_{eq}^2\tilde{C}_{2}^{\top}\tilde{C}_{2}, \\
\int_{\mathbb{R}^{3}}\tilde{D}_{3}\frac{\partial \tilde{D}}{\partial x_3}x^{\top}\tilde{p}_{eq}\td x =  \tilde{\sigma}_{eq}^4 \left(\begin{array}{c c c} \tilde{B}_1\tilde{B}_3 & 0 & 0 \\ 0 & \tilde{B}_2\tilde{B}_3 & 0 \\ 0 & 0 & 0 \end{array} \right)+\tilde{\sigma}_{eq}^2\tilde{C}_{3}^{\top}\tilde{C}_{3},
\end{split}
\end{equation*}
respectively.
Thus,
\begin{equation*}
\hat{k}''_{x}(0)=\frac{\tilde{\sigma}_{eq}^{2}}{\sigma_{eq}^{2}}\left[\tilde{\sigma}_{eq}^{2}\left(\begin{array}{c c c} \tilde{B}_1\tilde{B}_2+\tilde{B}_1\tilde{B}_3 & 0 & 0 \\ 0 & \tilde{B}_1\tilde{B}_2+\tilde{B}_2\tilde{B}_3 & 0 \\ 0 & 0 & \tilde{B}_1\tilde{B}_3+\tilde{B}_2\tilde{B}_3 \end{array} \right)+\tilde{C}^{2}\right].
\end{equation*}
Furthermore, since $\sum \tilde{B}_i=0$, we arrive at,
\begin{equation}
\hat{k}''_{x}(0)=\frac{\tilde{\sigma}_{eq}^{2}}{\sigma_{eq}^{2}}[-\tilde{\sigma}_{eq}^{2} \text{diag}(\tilde{B}_1^2, \tilde{B}_2^2, \tilde{B}_3^2 )+\tilde{C}^2],\nonumber
\end{equation}
which is the stated result.

\section*{Appendix B: The computation of $M_4$ and $M_5$ for the Langevin dynamics model}

\blue{
Here we  show how to obtain the formulas for $\hat k^{(4)}_A(0)$ and $\hat k^{(5)}_A(0)$ for the Langevin dynamics model. Recall that the generator is given by,  
\begin{equation}
\tilde{\mathcal{L}} =  v\frac{\partial}{\partial x} + (-U'(x)-\tilde{\gamma}v) \frac{\partial}{\partial v} + \tilde\gamma k_B\tilde{T} \frac{\partial^2}{\partial v^2}.\nonumber
\end{equation}
Furthermore, we have the response operator,
\begin{equation}
\hat{k}_{A}(t) = \frac{1}{k_B T}\int_{\mathbb{R}^2}\int_{\mathbb{R}^2} 
v u \tilde{p}(x,v,t|y,u,0) \td x\td v\td y \td u.\nonumber
\end{equation}
Direct calculations yield,
\begin{equation*}
\begin{split}
& \hat{k}_{A}'(t)
=  \frac{1}{k_B T}\int_{\mathbb{R}^2}\int_{\mathbb{R}^2}
(-U'(x)-\tilde \gamma v)u \tilde{p}(x,v,t|y,u,0) \td x\td v\td y\td u,\\
& \hat{k}_{A}''(t)= \frac{1}{k_B T}\int_{\mathbb{R}^2}\int_{\mathbb{R}^2}
(-vU''(x)+\tilde \gamma U'(x)+\tilde \gamma^{2} v) u
 \tilde{p}(x,v,t|y,u,0) \td x\td v\td y \td u.
\end{split}
\end{equation*}

From fitting $M_0$, we have $T=\tilde{T}.$ Now we applying the generator $\tilde{\mathcal{L}}$ again and we obtain
\begin{equation*}
    \hat k'''_A=\frac{1}{k_B T}\int (-v^2 U'''(x)+2\tilde{\gamma}v U''(x)+U'(x)U''(x)-\tilde{\gamma}^{3}U'(x)-\tilde{\gamma}^3v)u \tilde{p}(x,v,t|y,u,0)\td x \td v\td y\td u,
\end{equation*}
which leads to $\hat k'''_A(0)=2\tilde{\gamma}\mathbb{E}_{eq}(U''(x))-\tilde{\gamma}^3$. And applying the generator again we have
\begin{equation*}
    \hat k^{(4)}_A=\frac{1}{k_B T}\int f(x,v)u \tilde{p}(x,v,t|y,u,0)\td x \td v\td y\td u,
\end{equation*}
where $
f(x,v)=-v^3U^{(4)}(x)+4\tilde{\gamma}v^2U'''(x)+v(U''(x))^{2}+3vU'(x)U'''(x)-2\tilde{\gamma}k_B\tilde TU'''(x)-2\tilde{\gamma}U'(x)U''(x)-3\tilde{\gamma}^{2}vU''(x)+\tilde{\gamma}^{3}U'(x)+\tilde{\gamma}^{4}v$.
This formula leads to
\begin{equation*}
\hat k^{(4)}_A(0)=-\frac{1}{k_B\title{T}}\mathbb{E}_{eq}(v^4)\mathbb{E}_{eq}U^{(4)}(x)+\mathbb{E}_{eq}((U''(x))^2)+3\mathbb{E}_{eq}(U'(x)U'''(x))-3\tilde{\gamma}^2\mathbb{E}_{eq}(U''(x))+\tilde{\gamma}^4.
\end{equation*}
For $\hat k_A^{(5)}(0)$ we have
\BEA
\hat k_A^{(5)}(0)&=&\frac{7\tilde{\gamma}}{k_BT}\mathbb{E}_{eq}(v^4)\mathbb{E}_{eq}U^{(4)}(x)-8\tilde{\gamma}k_BT\mathbb{E}_{eq}(U^{(4)}(x))-3\tilde{\gamma}\mathbb{E}_{eq}((U''(x))^2) \nonumber\\ &&-13\tilde{\gamma}\mathbb{E}_{eq}(U'(x)U'''(x))+4\tilde{\gamma}^3\mathbb{E}_{eq}(U''(x))-\tilde{\gamma}^5.\nonumber
\EEA
Thus these formulas require the fourth moment of $v$, $\mathbb{E}_{eq}(U'U''')$, $\mathbb{E}_{eq}(U^{(4)})$, $\mathbb{E}_{eq}((U'')^2)$ and $\mathbb{E}_{eq}(U'')$, in order to compute $\hat k_A^{(4)}(0)$ and $\hat k_A^{(5)}(0)$.

Notice $v\sim\mathcal{N}(0,k_{B}T)$, thus $\mathbb{E}_{eq}(v^4)=3(k_B  T)^2$. Therefore, we can rewrite the two formula into
\begin{equation}
\begin{aligned}
\hat k^{(4)}_A(0)=&-3k_BT\mathbb{E}_{eq}U^{(4)}(x)+\mathbb{E}_{eq}((U''(x))^2)+3\mathbb{E}_{eq}(U'(x)U'''(x))-3\tilde{\gamma}^2\mathbb{E}_{eq}(U''(x))+\tilde{\gamma}^4, \\
\hat k_A^{(5)}(0)=&13\tilde{\gamma}k_BT\mathbb{E}_{eq}(U^{(4)}(x))-3\tilde{\gamma}\mathbb{E}_{eq}((U''(x))^2)-13\tilde{\gamma}\mathbb{E}_{eq}(U'(x)U'''(x))+4\tilde{\gamma}^3\mathbb{E}_{eq}(U''(x))-\tilde{\gamma}^5.
\end{aligned}\nonumber
\end{equation}
These are the formulas implemented in producing the results in Fig. 2.

}

%\bibliographystyle{plain}
%\bibliography{ref}

%%%%%%%%%%%%%%%%%%%%%%%%%
%%%%%%%%%%%%%%%%%%%%%%%%%

\end{document}